\documentclass[11pt,reqno]{amsart}

\usepackage{hyperref}
\hypersetup{colorlinks=true, }
\hypersetup{
  pdftitle={Rectifiability and upper Minkowski bounds for singularities of harmonic Q-valued maps},
  pdfauthor={Camillo de Lellis, Andrea Marchese, Emanuele Spadaro and Daniele Valtorta},
  pdfsubject={},
  pdfkeywords={},
  pdfpagelayout=SinglePage,
  pdfpagemode=UseOutlines,
  colorlinks,
  citecolor=blue,
  bookmarksopen,
  linkcolor=[rgb]{0,0,0.7},
  urlcolor=[rgb]{0,0,0.4},
  citecolor=[rgb]{0.4,0.1,0}
}

\usepackage{graphicx}
\usepackage{xfrac}
\usepackage{esint,amsfonts,amsmath,amssymb,epsfig,mathrsfs,bm,accents,verbatim}
\usepackage{ifthen,epic,eepic,color,stmaryrd,yfonts,mathtools}
\usepackage{a4wide,times}

\usepackage{enumerate}
\begingroup
\newtheorem{theorem}{Theorem}[section]
\newtheorem{lemma}[theorem]{Lemma}
\newtheorem{proposition}[theorem]{Proposition}

\endgroup

\theoremstyle{definition}
\newtheorem{definition}[theorem]{Definition}
\newtheorem{ipotesi}[theorem]{Assumption}

\theoremstyle{remark}
\newtheorem{remark}[theorem]{Remark}

\numberwithin{equation}{section}

\newcommand{\supp}{\mathrm{spt}}
\newcommand{\R}{\mathbb{R}}

\newcommand{\dist}{{\textup {dist}}}

\renewcommand{\d}{{\rm d}}
\newcommand{\spt}{{\rm spt}\,}
\newcommand{\res}{\mathop{\hbox{\vrule height 7pt width .5pt depth 0pt
\vrule height .5pt width 6pt depth 0pt}}\nolimits}
\newcommand{\loc}{\textup{loc}}
\renewcommand{\and}{\quad \text{and} \quad}
\def\a#1{\left\llbracket{#1}\right\rrbracket}

\newcommand{\cG}{{\mathcal{G}}}

\newcommand{\cH}{{\mathcal{H}}}

\newcommand{\mE}{{\mathcal{E}}} %replaceable, if \mathcal{E}_4 und \mathcal{E}_5 are ugly symbols

\newcommand{\B}[2]{B_{#1}\ton{#2}}
\newcommand{\ps}[2]{\left\langle#1,#2\right\rangle}
\newcommand{\ton}[1]{\left(#1\right)}
\newcommand{\qua}[1]{\left[#1\right]}
\newcommand{\cur}[1]{\left\{#1\right\}}
\newcommand{\abs}[1]{\left|#1\right|}

\newcommand\Z{{\mathbb Z}}

\newcommand\N{{\mathbb N}}

\def\spanna{{\operatorname{span}}}
\def\freq{{I_\phi}}
\def\I#1{{\mathcal{A}}_{#1}}

\newcommand{\Iqs}{{\mathcal{A}}_Q(\R^{n})}

\def\a#1{\left\llbracket{#1}\right\rrbracket}

\newcommand{\D}{\textup{Dir}}
\newcommand{\de}{\partial}

\newcommand{\etaa}{{\bm{\eta}}}

\renewcommand{\sfrac}[2]{\frac{#1}{#2}}

\def\sfava#1{{#1}}
\newcommand{\Addresses}{{% additional braces for segregating \footnotesize
  \bigskip
  \footnotesize

  C. D.L., A. M. and D. V., \ \textsc{Universit\"at Z\"urich, Winterthurerstrasse 190, CH-8057 Z\"urich (CH)}
  \par\nopagebreak
  
  E.S., \  \textsc{Universit\"at Leipzig, Augustusplatz 10 Leipzig (EU)}
  
  \par\nopagebreak
  
  \textit{E-mail address}, C. D.L.: \texttt{camillo.delellis@math.uzh.ch}

  \medskip

  \textit{E-mail address}, A. M.: \texttt{andrea.marchese@math.uzh.ch}

  \medskip

  \textit{E-mail address}, E. S.: \texttt{spadaro@math.uni-leipzig.de}

  \medskip

  \textit{E-mail address}, D. V.: \texttt{daniele.valtorta@math.uzh.ch}

}} 

\title[Rectifiability of the singular set
of $Q$-maps]{Rectifiability and upper Minkowski bounds for singularities of harmonic $Q$-valued maps}

\author{Camillo De Lellis, Andrea Marchese, Emanuele Spadaro and Daniele Valtorta}
\address{}
\curraddr{}
\email{}

\thanks{}



\date{}

\begin{document}

\begin{abstract}
In this article we prove that the singular set of Dirichlet-minimizing $Q$-valued functions is countably $(m-2)$-rectifiable and we give upper bounds for the $(m-2)$-dimensional Minkowski content of the set of singular points with multiplicity $Q$. 
\end{abstract}

\maketitle

%%%%%%%%%%%%%%%%%%%%%%%%%%%%%%%%%%%%%%%%%%%%%%%%%%%%%%%%%%%%%%%%%
%
%	INTRODUCTION
%
%%%%%%%%%%%%%%%%%%%%%%%%%%%%%%%%%%%%%%%%%%%%%%%%%%%%%%%%%%%%%%%%%
\section{Introduction}

$Q$-valued functions were introduced by Almgren in \cite{almgren_big} in order to model branching singularities of area minimizing currents in higher codimension. Indeed, it was first noticed by De Giorgi in his pioneering work \cite{DeGiorgi4} that an area minimizing hypersurface can be very well-approximated by the graph of a harmonic function if it is sufficiently close (in a weak sense) to a Euclidean plane.
In higher codimension, this statement is not true anymore at points of high multiplicity as it is well known that area minimizing surfaces can have branching singularities, cf. \cite[Section 5.2]{Survey1}. Almgren introduced a suitable notion of Dirichlet energy for functions taking a fixed number $Q$ of values in order to approximate efficiently area-minimizing currents in a neighborhood of a singular point of branching type with multiplicity $Q$. He then showed that ``harmonic'' (namely Dirichlet minimizing) $Q$-valued maps might be singular but the codimension of their singular set is at least $2$. In turn his monograph \cite{almgren_big} used such regularity property as a starting point to show that the Hausdorff dimension of the singular set of $m$-dimensional area-minimizing currents is at most $m-2$: in a nutshell Almgren's program in \cite{almgren_big} is a (fairly complicated) linearization procedure which reduces the bound on the dimension of the singular set for an area minimizing current to the same bound for the singular set of harmonic multivalued maps (cf. \cite{Survey1,Survey2} for a more precise description of Almgren's program which follows the recent approach
of \cite{DS0,DS2,DS3,DS4,DS5}). 

In this note we establish a more refined regularity property for the singular set of Dirichlet minimizing $Q$-valued functions on an $m$-dimensional domain, showing that indeed it is $(m-2)$-rectifiable (and hence $\mathcal{H}^{m-2}$ $\sigma$-finite). The latter property has already been shown by Krummel and Wickramasekera in \cite{KruWic} when $Q=2$ and  the same authors have announced that their proof can be extended to any $Q$, cf. \cite{KruWic2}. Our argument is however different, since it is based on the techniques introduced recently by Aaron Naber and the fourth author in \cite{NV}, whereas \cite{KruWic} draws on the approach of Simon (cf. for instance \cite{Simon}). Thus a byproduct of our proof is the additional information that the subset of singular points with highest multiplicity has locally finite Hausdorff $(m-2)$-dimensional measure (indeed it is possible to give an upper bound for its Minkowski $(m-2)$-dimensional content). On the other hand Krummel and Wickramasekera, adapting the techniques of Simon, obtain different byproducts, most notably the uniqueness of the tangent functions at $\mathcal{H}^{m-2}$-a.e. point and, for $Q=2$ and in the neighborhood of some special singular points, higher regularity of the singular set, cf. Remark \ref{r:KW}, \cite[Theorem C]{KruWic} and \cite{Krummel}. Of course, in view of Almgren's program, rectifiability results might be the starting point for a refined study of the singular set of area-minimizing currents, possibly leading to a solution to \cite[Problem 5.3]{GMTprob}. 

\medskip
Aside from applications to minimal currents, this work and the techniques developed here to study problems with variable homogeneity can be adapted to different topics in mathematics, see for example the recent works on free boundary problems \cite{FoSpa}, liquid crystals \cite{Alper} and $\Z/2$ harmonic spinors \cite{Boyu}. We also mention the recent works on the non-continuous singularities for $Q$-valued harmonic maps in \cite{HSV}.

\medskip

$Q$-valued functions are simply functions taking values in the space of unordered $Q$-tuples of points in $\R^n$, which is denoted by $\Iqs$. Following Almgren's convention, we will denote a point $T\in \Iqs$ as $T=\sum_{i=1}^Q \a{P_i}$, where $\a{P_i}$ is the Dirac measure concentrated on $P_i\in \R^n$. This space can be endowed with a natural distance given by
\begin{gather}
 d(T_1,T_2) =d\ton{\sum_{i=1}^Q \a{P_i},\sum_{i=1}^Q \a{S_i}}= \min_{\sigma\in \mathcal P_Q} \sqrt{\sum_{i=1}^Q \abs{P_i-S_{\sigma(i)}}^2 }\, ,
\end{gather}
where $\mathcal P _Q$ is the group of permutations of $Q$ elements. With this distance, $\Iqs$ is a complete metric space.
For a domain $\Omega\subseteq \R^m$, the Dirichlet energy and the space $W^{1,2} (\Omega, \Iqs)$ are defined in \cite{almgren_big} following a rather involved, albeit natural, geometric procedure (cf. \cite[Section 7.3]{Survey1}). It has been noticed in \cite{DS0} that modern analysis in metric spaces can be used to give an intrinsic simple definition of both objects. We refer to \cite{DS0,almgren_big} for a more detailed description of the space of $Q$-valued functions and Dirichlet minimizers, here we simply recall that Dirichlet minimizers are H\"{o}lder continuous functions with exponent $\alpha=\alpha(m,n,Q)$.

A point $x\in \Omega$ is a regular point for a $Q$-valued Dirichlet minimizer $u$ if there exists a neighborhood $B$ of $x$ and $Q$ harmonic functions $u_i: B \to \R^n$ such that for all $y\in B$:
\begin{gather}
 u(y) = \sum_{i=1}^Q \a{u_i(y)}\, ,
\end{gather}
and either $u_i(y)\neq u_j(y)$ for all $y\in B$, or $u_i\equiv u_j$. The complement of regular points are the singular points of $u$, denoted by $\Sigma_u$. Note that this set is automatically a closed set. Moreover, the main result regarding $Q$-valued functions in \cite{almgren_big} is that the Hausdorff dimension of $\Sigma_u$ is bounded from above by $m-2$. In particular:
\begin{theorem}[{\cite{almgren_big}, and \cite[Proposition 3.22]{DS0} }]\label{t:dim<=m-2}
If $u$ is a Dirichlet-minimizing $Q$-valued function $u:\Omega\subseteq \R^n \to \Iqs$, then $\Sigma_u$ is a relatively closed subset of $\Omega$ with Hausdorff dimension no larger than $m-2$. 
\end{theorem}

An important subset of $\Sigma_u$ consists of those singular points where all the values of $u(x)$ coincide, in other words
\begin{gather}
 \Delta_Q = \cur{x\in \Sigma_u \ \ \text{s.t.} \ \ u(x)=Q\a{P} \ \ \text{for some } \ P\in \R^n}\, .
\end{gather}
By H\"{o}lder regularity of the functions $u$, also the set $\Delta_Q$ is closed.

The main result of this note is then the following theorem. In the rest of the paper we will use the notation $B_r (E)$ for the open $r$-tubular neighborhood of the set $E$, namely
$B_r (E)= \{p: {\rm dist}\, (p, E)<r\}$. 

\begin{theorem}\label{t:Q-rett}
Let $u:\Omega\subseteq \R^m \to \Iqs$ be a Dirichlet minimizing function. Then for any compact set $K$ of $\Omega$, $\cH^{m-2} (\Delta_Q\cap K) < \infty$, and indeed we have the stronger
Minkowski-type estimate 
\begin{equation}\label{e:Minkowski}
|B_r (\Delta_Q) \cap K|\leq C (K, u) r^2 \qquad \forall r<1\, .
\end{equation}
Moreover
$\Delta_Q$ is $(m-2)$-countably rectifiable, namely it can be covered by countably many $C^1$ surfaces of dimension $m-2$,
except for a set of $\cH^{m-2}$ measure zero. 
\end{theorem}
As an immediate corollary of the latter statement we obtain:
\begin{theorem}\label{t:all-rect}
 The singular set $\Sigma_u$ of a Dirichlet minimizer $Q$-valued function $u$ is $(m-2)$-countably rectifiable.
\end{theorem}

\section*{Acknowledgments}
C. D.L. and A. M. were supported by ERC grant ``Regularity of area-minimizing currents'' (306247).

D. V. has been supported by SNSF grants 200021\_159403/1 and PZ00P2\_168006.

E. S. was supported by ERC-STG Grant n. 759229: HiCoS ``Higher Co-dimension Singularities: Minimal Surfaces and 
the Thin Obstacle Problem''.

\section{Main statements and plan of the paper}
\subsection{Preliminaries}
Before going into details, we want to underline again that for the reader who is inexperienced with $Q$-valued functions, a complete and readable introduction can be found in \cite{DS0}. In what follows for the values of the function $u$ we will use the notation
$u (x) = \sum_i \a{u_i (x)}$ and $ Du (x) = \sum_i \a{Du_i (x)}$. We refer the reader to \cite{DS0} for all the conventions
and terminologies.

In this section, we gather some preliminary results that will allow us to reduce our main theorems to a simpler version. First of all, we show how Theorem \ref{t:all-rect} follows from Theorem \ref{t:Q-rett}.
\begin{proof}[Proof of Theorem \ref{t:all-rect}]
 The proof follows easily from an inductive argument in $Q$. Indeed, for $Q=1$ we clearly have no singular set at all. For $Q=2$, the whole singular set coincides with $\Delta_Q$, and thus this is a corollary of \ref{t:Q-rett}. For a given $Q^* \geq 3$ we assume by induction that the statement of the theorem holds for all $Q<Q^*$. We fix a Dirichlet minimizing $Q^*$-valued map on some open set $\Omega$ and let $\Sigma_u = \Delta_{Q^*} \cup \Sigma'_u$, where $\Sigma'_u = \Sigma_u \setminus \Delta_{Q^*}$. Thus $\Sigma'_u$ is a relatively closed subset of the open set $\Omega' = \Omega \setminus \Delta_{Q^*}$. In particular, for all $x\in \Sigma'_u$, we have
 \begin{gather}
  u(x) = \sum_{i=1}^{Q^*} \a{P_i}\, ,
 \end{gather}
where at least one pair $\{P_i, P_j\}$ consists of different points. By H\"{o}lder continuity of $u$, there exists a neighborhood $B$ of $x$ and two multiple valued functions $u_1$ and $u_2$ such that $u_1$ has $Q_1$ values, $u_2$ has $Q_2$ values, $Q_1+Q_2=Q^*$ $Q_1\geq 1$ ,$Q_1\geq 1$, $Q_2\geq 1$ and
\begin{gather}
 u|_B= u_1 + u_2\, .
\end{gather}
Moreover, the images of $u_1$ and $u_2$ are disjoint. Thus $\Sigma_u \cap B$ is contained in the union of the singular sets of $u_1$ and $u_2$, which are $(m-2)$-rectifiable by inductive assumption. By a straightforward covering, this implies that $\Sigma'_u$ is $(m-2)$-rectifiable as well.
The rectifiability of $\Sigma_u$ follows now from the $(m-2)$-rectifiability of $\Delta_Q$.
\end{proof}

Thus, from now on we will focus just on the set of $Q$-points $\Delta_Q$. Before going further we state a useful simplification of our problem. Consider the function $\etaa:\Iqs \to \R^n$ defined by taking the average of the $Q$-tuple $T$, i.e.,
\begin{gather}
 \etaa(T):= \etaa\ton{\sum_{i=1}^Q \a{P_i}} = \frac 1 Q \sum_{i=1}^Q P_i\, .
\end{gather}
Note that this is a well-defined function on $\Iqs$, since its value is independent of the ordering in the $Q$-tuple $T$. It is useful to notice (see \cite[Lemma 3.23]{DS0}) that if $u$ is a Dirichlet-minimizer, then so is $\etaa \circ u$, thus in particular this is a classical harmonic function. Moreover, see again \cite[Lemma 3.23]{DS0}, if we introduce the map
\[
u' (x) = \sum_i \a{u_i (x) - \etaa\circ u}
\]
then $u'$ is again a Dirichlet-minimizer, and it satisfies the additional ``balancing condition'' $\etaa\circ u'\equiv 0$. Note that the singular points of $u$ coincide with the singular points of $u'$, and thus for the purposes of this article we can assume for simplicity and without loss of generality that $\etaa\circ u=0$. Note that under such assumption $\Delta_Q\subset \{x: u(x) = Q \a{0}\}$. However, \cite[Proposition 3.23]{DS0} delivers the following stronger information:
\begin{theorem} If $\Omega\subseteq \R^m$ is connected and $u: \Omega \to \Iqs$ is a Dirichlet minimizing map, then either $u \equiv Q \a{\etaa\circ u}$ or 
$\Delta_Q= \{x: u(x) = Q \a{0}\}$ and has Hausdorff dimension at most $m-2$. 
\end{theorem}

Therefore we can from now on assume, without loss of generality, that the following holds

\begin{ipotesi}\label{a:primaria}
$\Omega$ is a convex open subset of $\mathbb R^m$, $u: \Omega \to \Iqs$ is a minimizer of the Dirichlet energy with
$\etaa \circ u \equiv 0$ and positive Dirichlet energy. In particular 
\begin{equation}\label{e:Qpunti=0punti}
\Delta_Q = \{x: u(x) = Q \a{0}\}\,
\end{equation} 
and that $\Delta_Q$ is a strict subset of $\Omega$. 
\end{ipotesi}

\subsection{Frequency function and main steps}
Theorem \ref{t:Q-rett} will be split into two separate steps, namely the upper Minkowski estimate (Theorem \ref{t:Minkio}) and the rectifiability (Theorem \ref{t:Rect}), proved in the last two sections. In order to state the two steps, we need to introduce some notation and terminology.

For every $z \in \R^m$, we set $\nu_z:\R^m\setminus \{z\} \to \mathbb S^{m-1}$
given by $\nu_z(y):=\frac{y-z}{|y-z|}$. $D(x,r)$ denotes the Dirichlet energy of $u$ on the ball $B_r (x)$:
$\int_{B_r(x)}|D u|^2$. The height function $H (x,r)$ and Almgren's frequency function $I (x,r)$ are defined as 
$H (x,r) := \int_{\de B_{r}(x)} |u|^2$ and $I(x,r) := \frac{r D(x,r)}{H(x,r)}$.
In this paper we will however mainly work with a  ``smoothed'' version of $D$, $H$ and $I$, first introduced in \cite{DS5}. 

\begin{definition} Let $\phi$ be a Lipschitz nonincreasing function that is identically $1$ on $[0, \frac{1}{2}]$ and identically $0$ on $[1, \infty[$. The smoothed Dirichlet, height and frequency functions $D_\phi$, $H_\phi$ and $\freq$ are given, respectively, by
\begin{align}
D_\phi (x,r) := &\int |Du (y)|^2 \phi \left({\textstyle\frac{|y-x|}{r}}\right)\, dy\\
H_\phi (x,r) := &- \int  |u (y)|^2 |y-x|^{-1} \phi' \left({\textstyle\frac{|y-x|}{r}}\right)\, dy \\
\freq(x,r) := &\frac{r D_\phi(x,r)}{H_\phi (x,r)}
\end{align}
We also introduce
\begin{equation}
E_\phi (x,r) = - \int \left|\partial_{\nu_x} u (y)\right|^2 |y-x|  \phi' \left({\textstyle\frac{|y-x|}{r}}\right)\, dy\, .
\end{equation}
We omit  $x$ if it is the origin.
\end{definition}
Observe that, under Assumption \ref{a:primaria}, from Theorem \ref{t:dim<=m-2} we conclude that
$\Delta_Q$ is a set of measure zero in the ball $B_r (x)$, whenever $x\in \Omega$ and $r<  \dist (x, \de \Omega)$. Thus $H_\phi (x,r)$ is positive for every
such $x$ and $r$, which in turn implies that the frequency function is well defined for all such values. In some cases we will have to compute the above quantities for different functions $v$'s: we will then use the notation $D_{\phi, v} (x,r), H_{\phi, v} (x,r)$ and so on
to denote such dependence. The main tool of Almgren's regularity theory and of this paper is the monotonicity of the classical frequency function $I$ in the variable $r$. Almgren's computation can be easily extended to $I_\phi$ for any weight function $\phi$ as in the definition above
(a fact first remarked in \cite{DS5}). In particular both the classical frequency function and the smoothed ones can be defined at $r=0$ by taking the limit as $r\downarrow 0$.

In the rest of the paper we will often work under the following additional assumption.

\begin{ipotesi}\label{a:secondaria}
 $\Omega = B_{64} (0)$ and $\freq (64)\leq  \Lambda$. $\phi' (t) = - 2$ for every $t\in [\frac{1}{2}, 1]$ and $0$ otherwise. 
\end{ipotesi}

A simple covering argument allows then to recover Theorem \ref{t:Q-rett} from the following 

\begin{theorem}\label{t:Minkio}
 Under the Assumptions \ref{a:primaria} and \ref{a:secondaria} there is a constant $C= C (m,n,Q, \Lambda)$ such that
 \begin{gather}\label{eq_main_est}
|B_\rho (\Delta_Q) \cap B_{1/8} (0)| \leq C \rho^{2} \qquad \forall \rho>0\, .
 \end{gather}
\end{theorem}

\begin{theorem}\label{t:Rect}
 Under the Assumptions \ref{a:primaria} and \ref{a:secondaria} the set $\Delta_Q\cap B_{1/8} (0)$ is countably $(m-2)$-rectifiable. 
\end{theorem}

\begin{remark}\label{r:KW}
The singular set $\Delta_Q$ can be further subdivided according to the value of the frequency function $I (x,0)$, which must be positive at each singular $x$ (cf. Lemma \ref{l:bound_dal_basso}). For $Q=2$ the minimal value of $I (x, 0)$ at singular points is $\frac{1}{2}$ and the combination of the works \cite{KruWic} and \cite{Krummel} imply the real analiticity of $\Delta_2$ in a neighborhood of any such point. Moreover \cite{Krummel} shows the real analiticity of $\Delta_2\cap U$ in any open set $U$ for which the frequency function is constant on $\Delta_2\cap U$. 
\end{remark}

\subsection{Spines and pinching} Our proof is a nontrivial adaptation of the techniques of \cite{NV}. In particular, the main estimates will be derived from a Reifenberg-type result and estimates on the Jones' numbers of the sets $\Delta_Q$ and suitable discretizations of it.

The main ingredient is again the frequency function $I_\phi$. As mentioned above, for Dirichlet minimizers $I_\phi$ is a monotone function of $r$. The other impotant property is that $I_\phi$ controls the degree of homogeneity (or approximate homogeneity) of $u$. Indeed, $u$ is homogeneous of degree $\alpha$ at a point $x$ if and only if $I_\phi (x,r_1)=I_\phi (x,r_2)=\alpha$ for some $r_1<r_2$ (in which case it turns out that $r\mapsto I_\phi (x,r)$ is in fact constant). If $u$ were a classical function, its homogeneity would be equivalent to
\begin{gather}
 u(x+\lambda p)=\lambda^\alpha u(x+p) \quad \text{or}\quad \alpha u(x+p)=\ps{\nabla u(x+p)}{y}\, .
\end{gather}
From this formula, it is immediate to see that if $u$ is homogeneous of \textit{the same degree} $\alpha$ at two points $x\neq y$, then automatically $u$ is invariant with respect to the line joining $x$ and $y$. Indeed, we easily have
\begin{gather}\label{eq_homo_alpha}
 \ps{\nabla u(p)}{x-y}= \alpha u(p)-\alpha u(p)=0\, \qquad \mbox{for all $p\in \R^n$.}
\end{gather}
The same conclusions hold for $Q$-valued functions provided we introduce the correct terminology. 

If $u$ happens to be homogeneous with respect to some points $\cur{x_i}$ spanning a $k$-dimensional subspace, then $u$ is invariant with respect to this subspace. By Theorem \ref{t:dim<=m-2}, a $u$ which satisfies Assumption \ref{a:primaria} and is invariant with respect to an $m-1$ dimensional does not exist, thus must have empty $\Delta_Q$, thus making $m-2$ the maximum number of invariant directions that allow for some singular behaviour of $u$. Moreover, if $u$ has an invariant subspace of dimension $m-2$, then the singular set $\Delta_Q$ is either empty or it coincides with this subspace.

\medskip 

The monotonicity formula for $I_\phi$ gives a quantitative measurement (in an integral sense) of how close $u$ is to being homogeneous of degree $I_\phi$ at a point $x$. The precise statement can be found in Proposition \ref{p:pinching}. In turn this leads to the most important estimate of the note:

\begin{definition}
Let $u$ and $\phi$ be as in Assumptions \ref{a:primaria} and \ref{a:secondaria}. 
For every $x\in B_1$ and every $0<s\leq r \leq 1$ we let
\begin{equation}\label{e:W}
W^r_{s}(x) := \freq(x,r) - \freq (x,s)\, 
\end{equation}
be the ``pinching'' of the frequency function between the radii $s$ and $r$. 
\end{definition}

\begin{theorem}[{Cf. Theorem \ref{t:stima_fondamentale}}]\label{t:stima_fondamentale_2}
There exist $C_{\ref{e:stima_fondamentale}} = C_{\ref{e:stima_fondamentale}}(\Lambda,m,n,Q)>0$ such that, if $u$ and $\phi$ satisfy the Assumptions \ref{a:primaria} and \ref{a:secondaria}, $x_1, x_2 \in B_{1/8}(0)$ and $\abs{x_1-x_2}\leq r/4$, then
\begin{equation}
\big\vert \freq \big(z, r) - \freq \big(y, r)\big\vert \leq C_{\ref{e:stima_fondamentale}}\,
\sfava{\qua{\ton{W^{4r}_{r/8}(x_1)}^{\sfrac{1}{2}} + \ton{W^{4r}_{r/8}(x_2)}^{\sfrac{1}{2}}}} |z-y|
\qquad \forall z,y \in [x_1, x_2]\, .
\end{equation}
\end{theorem}

With the latter estimate we will be able to bound in a quantitative way the distance between $\Delta_Q\cap \B r x$ and a carefully chosen $m-2$ dimensional plane $L_{x,r}$ for all $x,r$ (cf. Section \ref{sec_best_app}). This, combined with an inductive covering of $\Delta_Q$ and the generalized Reifenberg theorem proved in \cite{NV}, will allow us to conclude the proof.

\subsection{Plan of the paper}
The rest of the note is organized as follows:

\begin{itemize}
\item Section \ref{s:identita} gives several important bounds and identities on the smoothed frequency function. In particular, Proposition \ref{p:monot} states the crucial monotonicity identities and the related computations used later; Lemma \ref{l:bound_dal_basso} shows a fundamental $\varepsilon$-regularity theorem, namely that $\freq (x,r)$ cannot go below a certain threshold when $x\in \Delta_Q$; Lemma \ref{l:limiti uniformi_II} gives useful bounds for the frequency and height function at different points and scales.
\item Section \ref{s:drop} gives the most important new ingredient of the paper, namely it proves Theorem \ref{t:stima_fondamentale_2}. Similar estimates are a fundamental starting point for the results of \cite{NV} on the rectifiability of the singular set for harmonic maps and are a direct consequence of the monotonicity formula. In our framework the proof is instead rather nontrivial.
\item Proposition \ref{t:stima_fondamentale} is used in Section \ref{sec_best_app} to show that the average of the frequency drop at scale $r$ with respect to a general measure $\mu$ controls the $(m-2)$-mean flatness of $\mu$, also called Jones' number $\beta_2$, cf. Proposition \ref{p:mean-flatness vs freq}. 
\item In turn, Proposition \ref{p:mean-flatness vs freq} is combined with the Reifenberg-type methods developed in \cite{NV} to prove the Minkowski
bound of Theorem \ref{t:Minkio}.
\item Finally, the Minkowski bounds and Proposition \ref{p:mean-flatness vs freq} allows a suitable estimate of average of the Jones' number of the measure $\mathcal{H}^{m-2} \res \Delta_Q$: the results of \cite{NV} and of \cite{AzzTol} characterize the rectifiability of $\mu$ in terms of such average and imply therefore directly Theorem \ref{t:Rect}. 
\end{itemize}

%%%%%%%%%%%%%%%%%%%%%%%%%%%%%%%%%%%%%%%%%%%%%%%%%%%%%%%%%%%%%%%%%
%
%	SECTION 1
%
%%%%%%%%%%%%%%%%%%%%%%%%%%%%%%%%%%%%%%%%%%%%%%%%%%%%%%%%%%%%%%%%%
\section{Smoothed frequency function and relevant identities}\label{s:identita}

\subsection{Properties of the frequency function}
We recall next the monotonicity identity for the smoothed frequency function, which is the counterpart of the monotonicity of Almgren's ``classical'' frequency function $I$, cf. \cite[Eq. (3.48)]{DS0}. The monotonicity of $\freq$ is contained in the arguments of \cite{DS5}, but since this is not
explicitly mentioned there, we provide here the relevant statements and the short proof. Moreover we will differentiate the functions also in the variable $x$. We summarize the relevant identities in the following Proposition. 

\begin{proposition}\label{p:monot}
Under Assumption \ref{a:primaria} we have that the functions $D_\phi$, $H_\phi$ and $\freq$ are $C^1$ in both variables. Moreover the following identities hold:
\begin{align}
D_\phi (x,r) =& - \frac{1}{r} \int \phi' \left({\textstyle\frac{|y-x|}{r}}\right)  \sum_{i=1}^Q \partial_{\nu_x} u_i (y) \cdot u_i (y)\, dy\label{e:calcolo1}\\
\partial_r D_\phi (x,r) =& \frac{m-2}{r} D_\phi (x,r) + \frac{2}{r^2} E_\phi (x,r)\label{e:calcolo2}\\
\partial_v D_\phi (r,x) = &- \frac{2}{r} \int  \phi' \left({\textstyle\frac{|y-x|}{r}}\right)  \sum_{i=1}^Q \partial_{\nu_x} u_i (y) \cdot \partial_v u_i (y)\, dy\label{e:calcolo3}
\end{align}
\begin{align}
\partial_r H_\phi (x,r) = & \frac{m-1}{r} H_\phi (x,r) + 2 D_\phi (x,r)\label{e:calcolo4}\\
\partial_v H_\phi (x,r) = & - 2 \int  \phi' \left({\textstyle\frac{|y-x|}{r}}\right) |y-x|^{-1}  \sum_{i=1}^Q u_i (y) \cdot \partial_v u_i (y)\, dy\label{e:calcolo5}\, .
\end{align}
In particular both $\freq (x,r)$ and $r^{1-m} H_\phi (x,r)$ are nondecreasing functions of $r$ and we have the following
identities
\begin{equation}\label{e:monotonia freq}
 \partial_r \freq (x,r)=\frac{2}{r H_\phi (x,r)^2} \Big(H_\phi (x,r) \; E_\phi (x,r) -
r^{2}D_\phi (x,r)^2 \Big) \geq 0\, 
\end{equation}
 \begin{equation}\label{eq_doubling}
  s^{1-m} H_\phi (x,s)= r^{1-m} H_\phi (x,r) \exp\ton{-2\int_s^r \freq (x,t) \frac{dt}{t}}\, .
 \end{equation}
\end{proposition}

\begin{remark}
Note that by letting $\phi \uparrow {\bf 1}_{[0,1[}$ we recover corresponding statements for the classical Dirichlet, height and frequency functions, at the price of a loss of smoothness: some of the identities are, in particular, true in a suitable a.e. sense. A particularly useful inequality that is instead valid for every $x, s$ and $r$ is the monotonicity 
\begin{equation}\label{e:H_mon}
s^{1-m} H (x,s) \leq r^{1-m} H (x,r) \qquad \forall 0  < s\leq r < \dist (x, \partial \Omega)\, .
\end{equation}
\end{remark}

\begin{proof} First of all we can assume, without loss of generality that $\phi$ is smooth: indeed in this case 
\begin{itemize}
\item the smoothness of $\freq$ in $r$ is an obvious consequence of the smoothness of $\phi$;
\item the smoothness of $\freq$ in $x$ follows from the usual fact that the convolution of a smooth kernel with an integrable function is smooth.
\end{itemize}
After having established the above identities for $\phi$ smooth we can
approximate any Lipschitz test with a sequence of bounded $\phi_k$ that are smooth, have uniformly bounded derivatives and converge strongly in $W^{1,p}$ for every $p<\infty$. It is then easy to see that $\partial_v D_{\phi_k}$ and $\partial_r H_{\phi_k}$ converge uniformly and to conclude in the limit the corresponding formulae. 
As already noticed $H_\phi$ is positive and thus $\freq$ is also $C^1$.

\medskip

\eqref{e:calcolo1} follows from testing \cite[Eq. (3.5)]{DS0} with the map
\[
\psi (y, u) := \phi \left({\textstyle\frac{|y-x|}{r}}\right) u\, .
\]
Differentiating in $r$ we get
\[
\partial_r D_\phi (x,r) = -\int |Du (y)|^2 {\textstyle{\frac{|y-x|}{r^2}}}  \phi' \left({\textstyle\frac{|y-x|}{r}}\right)\, dy\, .
\]
Testing \cite[Eq. (3.3)]{DS0} with the vector field
\[
\varphi (y) =  \phi \left({\textstyle\frac{|y-x|}{r}}\right) (y-x)
\]
we obtain \eqref{e:calcolo2}. Similarly, differentiating in $x$ we achieve
\[
\partial_v D_\phi (x,r) =  \int |Du (y)|^2  \phi' \left({\textstyle\frac{|y-x|}{r}}\right) {\textstyle{\frac{y-x}{r|y-x|}}} \cdot v\, dy
\]
and from the latter we derive \eqref{e:calcolo3} testing \cite[Eq. (3.3)]{DS0} with the vector field
\[
\varphi (y) = \phi \left({\textstyle\frac{|y-x|}{r}}\right) v\, .
\]

\medskip

Changing variables in the integral we rewrite the formula for the height in two different ways 
\begin{align}
H_\phi (x,r) = & - \int |u (x+z)|^2 |z|^{-1} \phi' \left({\textstyle\frac{|z|}{r}}\right)\, dz
=  - \frac{1}{r^{m-1}} \int |u (x+r\zeta)|^2 |\zeta|^{-1} \phi' (|\zeta|)\, d\zeta\label{e:cambio_variabili}
\end{align}
Next, since $u$ is a continuous $W^{1,2}$ map and $\Iqs \ni P \to |P|^2 = \sum_i P_i$ is a locally Lipschitz map, $|u|^2$ is indeed a $W^{1,2}_{\loc}$ map. Moreover the chain rule formulae \cite[Proposition 1.12]{DS0} imply
\begin{equation}\label{e:|u|der}
\partial_v |u|^2 (y) = 2 \sum_i u_i (y) \partial_v u_i (y)\, .
\end{equation}
We thus differentiate the first integral in \eqref{e:cambio_variabili} in $v$ and the second integral in \eqref{e:cambio_variabili} in $r$ to get
\begin{align}
\partial_v H_\phi (x,r) = & - 2 \int  |z|^{-1}  \phi' \left({\textstyle\frac{|z|}{r}}\right)\sum_i \partial_v u_i (x+z) \cdot u_i (x +z)\, dz\label{e:dvH}\\
\partial_r H_\phi (x,r) = & \frac{m-1}{r} H_\phi (x,r) - \frac{2}{r^{m-1}} \int |\zeta|^{-1} \phi' (|\zeta|) \sum_i \partial_\zeta u_i (x+r\zeta)
\cdot u_i (x+r\zeta)\, d\zeta\, \label{e:drH}\, .
\end{align}
Changing the integration variable back to $y$ in \eqref{e:dvH} we achieve \eqref{e:calcolo5}. Changing variable in \eqref{e:drH} we get
\[
\partial_r H_\phi (x,r) =  \frac{m-1}{r} H_\phi (x,r) - 2 \int \phi' \left({\textstyle\frac{|y-x|}{r}}\right) \sum_i \partial_{\nu_x} u_i (y) \cdot u_i (y)\, dy\, 
\]
and hence we conclude \eqref{e:calcolo4} from \eqref{e:calcolo1}. 

\medskip

The expression for $\partial_r \freq (x,r)$ in \eqref{e:monotonia freq} is an obvious consequence of \eqref{e:calcolo2} and \eqref{e:calcolo4}, whereas such expression turns out to be nonnegative using \eqref{e:calcolo1} and the Cauchy-Schwartz inequality:
\begin{align*}
r^2 D_\phi (x,r) = & \left( \int - \phi' \left({\textstyle\frac{|y-x|}{r}}\right)  \sum_i \partial_{\nu_x} u_i (y) \cdot u_i (y)\, dy\right)^2\\
\leq & \int - \phi' \left({\textstyle\frac{|y-x|}{r}}\right) |y-x|^{-1} \sum_i |u_i (y) |^2\, dy
\int - \phi'  \left({\textstyle\frac{|y-x|}{r}}\right) |y-x| \sum_i |\partial_{\nu_x} u_i (y)|^2\, dy\\
= & H_\phi (x,r) E_\phi (x,r)\, .
\end{align*}
Note that the assumption $-\phi' \geq 0$ is used crucially only in the inequality above. 

\medskip

Finally, we can rewrite \eqref{e:calcolo4} as 
\[
\partial_r \log \left(r^{1-m} H_\phi (x,r)\right) = \frac{\partial_r H_\phi (x,r)}{H_\phi (x,r)} - \frac{m-1}{r} = 2 \frac{D_\phi (x,r)}{H_\phi (x,r)}
= \frac{2}{r} \freq (x,r)\, .
\]
Integrating the latter identity we achieve \eqref{eq_doubling} and the monotonicity of $r^{1-m} H_\phi (x,r)$ follows from the positivity of $\freq$. 
\end{proof}

\subsection{\texorpdfstring{$\varepsilon$-regularity}{epsilon-regularity}}
The following lemma is, loosely speaking, an $\varepsilon$-regularity theorem that shows that, if  the  frequency is sufficiently small at a certain scale, there are no $Q$-points at a slightly smaller scale .
\begin{lemma}\label{l:bound_dal_basso}
There is a constant $0<\epsilon_{\ref{e:bound_dal_basso}} (m,n,Q)\leq 1$ with the following property. Under Assumption \ref{a:primaria},
\begin{equation}\label{e:bound_dal_basso}
\freq (x,r)\leq \epsilon_{\ref{e:bound_dal_basso}}  \quad \Longrightarrow \quad \Delta_Q\cap \B {r/4}{x}=\emptyset\, . 
\end{equation}
\end{lemma}
\begin{proof} 
Without loss of generality, we can assume $x=0$ and $r=1$. Suppose that $\freq (1)\leq 1$ and that there exists $y\in \Delta_Q\cap \B {1/4}{0}$. By \cite[Theorem 3.9]{DS0}, we have the existence of constants $\alpha (m, Q)>0$ and $C (m,n,Q)$ such that
\begin{equation}\label{e:holder}
 [u]_{C^{0, \alpha} (B_{1/4})} \leq C \ton{\int_{B_{1/2}} |Du|^2}^{\frac{1}{2}} \leq C D_\phi (1)^{\sfrac{1}{2}}\, .
\end{equation}
In particular, since $u (y) = Q \a{0}$ for some $y\in \B {1/4}{0}$, we have
\begin{equation}\label{e:da0a1/2}
\int_{\de B_{1/4}} |u|^2 \leq C D_\phi (1)\, .
\end{equation}
Note next that by passing in polar coordinates we use \eqref{e:H_mon} to derive
\[
H_\phi ({\textstyle\frac{1}{4}}) \leq C \int_{\de B_{1/4}} |u|^2 \leq C D_\phi (1)\, .
\]
By the growth estimates \eqref{eq_doubling}, since we assumed that $\freq (1)\leq 1$, we obtain
\begin{gather}
H_\phi (1) \leq C  H_\phi ({\textstyle\frac{1}{4}}) \leq C D_\phi (1)\, ,
\end{gather}
which immediately implies
\[
\freq (1)\geq C^{-1}\equiv \epsilon_{\ref{e:bound_dal_basso}} (m,n,Q)\, .\qedhere
\]
\end{proof}

\subsection{Elementary upper bounds}
We now prove that the value of $H_\phi$ (resp: $\freq$) at a point $x$, at a certain scale, gives a uniform upper bounds in a ball around $x$ on the same quantity at smaller scales.
\begin{lemma}\label{l:limiti uniformi_II}
There exists a constant $C(m, \phi)$ with the following property. If $u$ satisfies Assumption \ref{a:primaria}, then
\begin{align}
H_\phi (y, \rho) \leq & C H_\phi (x, 4\rho) \qquad &\forall y\in B_{\rho} (x) \subset B_{4\rho} (x) \subset \Omega\, ,\label{e:H_spostata}\\
\freq (y, r) \leq & C (\freq (x, 16r) +1) \qquad &\forall y \in B_{r/4} (x) \subset B_{16r} (x) \subset \Omega\, .\label{e:I_spostata}
\end{align}
\end{lemma}
\begin{proof}
 The proof is a standard computation, see for example \cite[theorem 2.2.8]{hanlin} in the case of harmonic functions and for the classical
frequency and height.

We first argue for \eqref{e:H_spostata} and assume, without loss of generality $x=0$ and $\rho=1$. Using \eqref{e:H_mon} we easily see that
\[
\int_{B_2} |u|^2 \leq C \int_{\partial B_r} |u|^2 \qquad \forall r\in ]2, 4[\, .
\]
Averaging the right hand side against the measure $-r^{-1} \phi' (r/4) dr$ and passing to polar coordinates we achieve
\[
\int_{B_2} |u|^2 \leq C H_\phi (4)\, .
\]
On the other hand, since $B_1 (y) \subset B_2$, it is obvious that $H_\phi (y, 1) \leq C \int_{B_2} |u|^2$. This shows $H_\phi (y,1) \leq C H_\phi (0,4)$ and completes the proof of \eqref{e:H_spostata}.

We next argue for \eqref{e:I_spostata} and assume, again, $x=0$ and $r =1$. 
\eqref{e:H_spostata}, \eqref{e:monotonia freq} and \eqref{eq_doubling} give
\begin{align*}
H_\phi (y, 4) \leq & C H_\phi (0, 16) \leq C e^{C \freq (0, 16)} H_\phi \left(0, {\textstyle{\frac{1}{4}}}\right) \leq C e^{C \freq (0,16)} H_\phi (y,1)\\
= &C H_\phi (y,4) \exp\ton{C \freq (0,16) - 2\int_1^4 \freq (y,t) \frac{dt}{t}}\, .
\end{align*}
Since $H_\phi (y,4)$ is positive, taking the logarithm we conclude
\[
2 \freq (y,1) \int_1^4\frac{dt}{t} \leq C (1+ \freq (16))\, .\qedhere
\]
\end{proof}

% THIS IS A DUPLICATE OF PROPOSITION \ref{p:pinching}
% We close this section with a rephrasing of Proposition \ref{p:monot}, which we record for future convenience.
% \begin{corollary}\label{cor_mE}
%  Let $u$ be a Dirichlet minimizing harmonic function with $I(0,2)\leq \Lambda$ and $H(0,1)=1$. Set
%  \begin{gather}
%   \mathcal{E}_i(z) := \frac{\de u_i}{\de \nu_{x_\ell}} (z) -  I(x_\ell,|z-x_\ell|)\,u_i(z)  \qquad \mbox{for $i\in \{1, \ldots Q\}.$}
%  \end{gather}
% Then
% \begin{gather}
%  \int_{\B{3/2}{x}\setminus \B{1/2}{x} }\sum_i \abs{\mathcal{E}_i(z)}^2\leq C(m,\Lambda)\qua{I(x,3/2)-I(x,1/2)}
% \end{gather}
% \end{corollary}
% \begin{proof}
% By definition, we have
%  \begin{gather}
%   \int_{\B{3/2}{x}\setminus \B{1/2}{x} }\sum_i \abs{\mathcal{E}_i(z)}^2 = \int_{1/2}^{3/2}dt \int_{\de \B{t}{x}}\abs{\frac{\de u_i}{\de \nu_{x_\ell}} (z) -  I(x_\ell,|z-x_\ell|)\,u_i(z)  }^2=\\
%   = \int_{1/2}^{3/2}dt \qua{E(x,t)-I(x,t)^2H(x,t)}=\int_{1/2}^{3/2}dt \frac{1}{H(x,t)}\qua{H(x,t)E(x,t)-t^2D(x,t)^2}\, .
%  \end{gather}
% The estimate now follows from the bound $tH(x,t)/2 \leq C(m,\Lambda)$ (see \eqref{eq_unif_bounds}) and \eqref{e:monotonia freq}.
% \end{proof}
%%%%%%%%%%%%%%%%%%%%%%%%%%%%%%%%%%%%%%%%%%%%%%%%%%%%%%%%%%%%%%%%%
%
%	SECTION 2
%
%%%%%%%%%%%%%%%%%%%%%%%%%%%%%%%%%%%%%%%%%%%%%%%%%%%%%%%%%%%%%%%%%
\section{Main estimate on the frequency pinching}\label{s:drop}

The main goal is to prove Theorem \ref{t:stima_fondamentale} below: this is the essential ingredient
that allows us to use the techniques of \cite{NV} in our framework and eventually conclude the $(m-2)$-rectifiability
and $\cH^{m-2}$-local finiteness of the set $\Delta_Q$. 

\begin{definition}
Let $u$ and $\phi$ be as in Assumptions \ref{a:primaria} and \ref{a:secondaria}. 
For every $x\in B_1$ and every $0<s\leq r \leq 1$ we let
\begin{equation}
W^r_{s}(x) := \freq(x,r) - \freq (x,s)\, 
\end{equation}
be the ``pinching'' of the frequency function between the radii $s$ and $r$. 
\end{definition}

The next theorem shows how the variations of the frequency in
nearby points are controlled by the pinching of the two points.

\begin{theorem}\label{t:stima_fondamentale}
There exist $C_{\ref{e:stima_fondamentale}} = C_{\ref{e:stima_fondamentale}}(\Lambda,m,n,Q)>0$ such that, if $u$ and $\phi$ satisfy the Assumptions \ref{a:primaria} and \ref{a:secondaria}, $x_1, x_2 \in B_{1/8}(0)$ and $\abs{x_1-x_2}\leq r/4$, then
\begin{equation}\label{e:stima_fondamentale}
\big\vert \freq \big(z, r) - \freq \big(y, r)\big\vert \leq C_{\ref{e:stima_fondamentale}}\,
\sfava{\qua{\ton{W^{4r}_{r/8}(x_1)}^{\sfrac{1}{2}} + \ton{W^{4r}_{r/8}(x_2)}^{\sfrac{1}{2}}}} |z-y|
\qquad \forall z,y \in [x_1, x_2]\, .
\end{equation}

\end{theorem}

A main ingredient in the proof  of the theorem will also play a fundamental role in the next estimate and for this reason we show it here.

\begin{proposition}\label{p:pinching}
There exist $C_{\ref{e:pinching}} = C_{\ref{e:pinching}} (\Lambda, m ,n, Q)>0$ such that, if $u$ and $\phi$ satisfy the Assumptions \ref{a:primaria}
and \ref{a:secondaria}, then, for every $x\in B_{1/8}$, 
\begin{equation}\label{e:pinching}
\int_{B_2 (x)\setminus B_{1/4} (x)} \sum_i |(z-x)\cdot Du_i (z) - \freq (x, |z-x|) u_i (z)|^2\, dz \leq C \sfava{W^4_{1/8} (x)}\, .
\end{equation}
\end{proposition}

\subsection{Intuition for the proof}
In order to get an intuition for the theorem, we explain briefly the underlying idea with an example. Let $h$ be a $Q$-valued function such that \sfava{$I(0,4)-I(0,1/8)=0$ and $I(x,4)-I(x,1/8)=0$}, where $x\in \B{1/8}{0}\setminus \{0\}$. For the sake of simplicity, one could assume here that $h$ is actually an harmonic function, thus smooth.

By unique continuation, we immediately get that the frequency $I$ is constant for all radia both at the origin and at $x$. Set $I(0,0)=d$ and $I(x,0)=d'$. Note that the two values may a priori be different, but we want to show that this is not the case.
The monotonicity formula for $I$ implies that $h$ is a $d$-homogeneous function wrt $0$ and $d'$-homogeneous wrt $x$. In other words for all $y\in \R^m$
\begin{gather}\label{eq_symmetry}
 \ps{Du(y)}{y}= d u(y)\, , \quad \ps{Du(y)}{y-x}=d' u(y)\, .
\end{gather}
By subtracting these two equations, we prove that
\begin{gather}\label{eq_Dxu}
 \ps{Du(y)}{x} = (d-d')u(y)\, .
\end{gather}
Consider the function $f(t)=\log(I(tx,1))$, then naively we can make use of the external variation formulas and write
\begin{gather}\label{eq_party}
 f'(t)= \frac{d}{dt} \log\ton{\frac{\int_{\B 1 0} \abs{Du}^2 }{\int_{\de \B 1 0} \abs u^2}} = \frac{\int_{\B 1 0} \ps{Du}{D_x Du} }{\int_{\B 1 0} \abs{Du}^2} - \frac{\int_{\de \B 1 0} u D_x u}{\int_{\de \B 1 0} u^2}=\\
 =\frac{\int_{\de \B 1 0} D_n u D_x u }{\int_{\de \B 1 0} u D_n u} - \frac{\int_{\de \B 1 0} u D_x u}{\int_{\de \B 1 0} u^2}\, ,
\end{gather}
where we used without proper justification the integration by parts for $Q$-valued functions. By \eqref{eq_Dxu}, we have 
\begin{gather}\label{eq_f'}
 f'(t)=(d-d')-(d-d')=0\, ,
\end{gather}
which in turn implies that $f(0)=f(1)$, and so $d=d'$.

Theorem \ref{t:stima_fondamentale} is the quantitative version of this statement. For its proof, we will use the quantitative version of \eqref{eq_symmetry}, which is given in Proposition \ref{p:monot}. 

\subsection{Proof of Proposition \ref{p:pinching}} Assume $H_\phi(1)=1$. Using Proposition \ref{p:monot} we can compute
\begin{align}
& W_{1/4}^4 (x) =  \int_{1/4}^4 \partial_r \freq  (x, \tau)\, d\tau = \int_{1/4}^4 2 (\tau H_{\phi} (x, \tau))^{-1} (E_\phi (x,\tau) - \tau \freq (x ,\tau) D_\phi (x, \tau))\, d\tau\nonumber\\
= & \int_{1/4}^4 2 (\tau H_{\phi} (x, \tau))^{-1} (E_\phi (x,\tau) - 2\tau \freq (x,\tau) D_\phi (x,\tau) +  \sfava{\freq (x ,\tau)^2 H_\phi (x, \tau)})\, d\tau\nonumber\\
= & \int_{1/4}^4 2 (\tau H_{\phi} (x, \tau))^{-1} \int - \phi' \left({\textstyle{\frac{|y-x|}{\tau}}}\right) |y-x|^{-1}\nonumber\\
& \qquad\qquad\qquad \Big(|\partial_{\eta_{x}} u|^2 - 2 \freq (x, \tau) \sum_i \partial_{\etaa_{x}} u_i \cdot u_i + \freq (x, \tau)^2 |u|^2\Big)dy\, d\tau\notag\\
= & \sfava{\int_{1/4}^4 2 (\tau H_{\phi} (x, \tau))^{-1} \int - \phi' \left({\textstyle{\frac{|y-x|}{\tau}}}\right) |y-x|^{-1}
\underbrace{\sum_i \big|(y-x)\cdot Du_i (y) - \freq (x, \tau) u_i (y)\big|^2}_{=: \xi (y, \tau)}\, dy\, d\tau}\, .
\label{e:Fubini}
\end{align}
\sfava{Observe that $\phi' = - 2 {\bf 1}_{[1/2, 1]}$. Hence the integrand in \eqref{e:Fubini} vanishes outside $\{\frac{1}{2} \tau \leq |y-x| \leq \tau\}$
and considering that the integral in $\tau$ takes place on the interval $[\frac{1}{4}, 4]$, we can assume $\frac{1}{8} \leq |y-x|\leq 4$.  Next we introduce the function
\[
\zeta (y) := \sum_i \big|(y-x)\cdot Du_i (y) - \freq (x, |y-x|) u_i (y)\big|^2\, 
\]
and, using the observation above, the monotonicity of $I_\phi (x, \cdot)$ and the triangle inequality, we conclude
\[
\zeta (y) \leq 2 \xi (y, \tau) + 2 |I_\phi (x, \tau) - I_\phi (x, |y-x|)| |u (y)|^2 \leq 2 \xi (y, \tau) + 2 W^4_{1/8} (x) |u (y)|^2\, .
\]
Inserting the latter inequality in \eqref{e:Fubini} we infer
\begin{align}
W_{1/4}^4 (x) \geq &\int_{1/4}^4(\tau H_{\phi} (x, \tau))^{-1} \int  - \phi' \left({\textstyle{\frac{|y-x|}{\tau}}}\right) |y-x|^{-1} \zeta (y) dy\, d\tau\notag\\
&- 2 W^4_{1/8} (x) \int_{1/4}^4(\tau H_{\phi} (x, \tau))^{-1} \int  - \phi' \left({\textstyle{\frac{|y-x|}{\tau}}}\right) |y-x|^{-1} |u (y)|^2\, dy\, d\tau\notag\\
\geq& \int_{1/4}^4(\tau H_{\phi} (x, \tau))^{-1} \int  - \phi' \left({\textstyle{\frac{|y-x|}{\tau}}}\right) |y-x|^{-1} \zeta (y) dy\, d\tau - 8 W^4_{1/8} (x)\, .\label{e:Fubini_s}
\end{align}
}
Next, using \eqref{e:I_spostata} we conclude $\freq(x,\tau) \leq C$ for every $\tau\leq 4$ and we can therefore use \eqref{e:H_spostata} and \eqref{eq_doubling} (together with $H_\phi (1) =1$) to find a uniform bound from below for $H_\phi (x, \tau)$ when $\tau\in [1/4, 4]$.
\sfava{Hence, from \eqref{e:Fubini_s}}
\[
C \sfava{W^4_{1/8} (x)} \geq \int \zeta (y) \underbrace{\int_{1/4}^4  - \phi' \left({\textstyle{\frac{|y-x|}{\tau}}}\right) |y-x|^{-1}d\tau}_{=: M (y)}\, dy\, .
\]
Since $\phi' = - 2 {\bf 1}_{[1/2, 1]}$ we can explicitly compute
\[
M (y) = \frac{2}{|y-x|} \left[ \min \{4, 2 |y-x|\} - \max \{{\textstyle{\frac{1}{4}}}, |y-x|\}\right] \geq 2{\bf 1}_{B_2 (x)\setminus B_{1/4} (x)} (y) \, ,
\]
which clearly completes the proof. 
\subsection{Proof of Theorem \ref{t:stima_fondamentale}} Without loss of generality, we assume $r=1$ and $H_\phi (1)=1$. For simplicity, we fix the notation
\begin{gather}
 W(x) :=\sfava{W^{4}_{1/8}(x)= \freq (x,4)-\freq (x,1/8)}\, 
\end{gather}
and we introduce the measure
\[
\mu_x := - |y-x|^{-1} \phi' \left(|y-x|\right)\, dy 
\]
and the vectors
\[
\eta_x (y) := y-x = |y-x| \nu_x (y)\qquad v:= x_2-x_1\, .
\]
Combining \eqref{e:calcolo3} and \eqref{e:calcolo5}, we deduce 
\begin{align}\label{e:derivata}
\partial_v \freq (x,1)
= & 2 H_\phi (x,1)^{-1} \left[ \int \sum_i \partial_v u_i \cdot \de_{\eta_x} u_i \, d\mu_x - 
\freq (x,1) \int \sum_i u_i\cdot \partial_v u_i  d\mu_x \right]\, .
\end{align}
Let
\[
\mathcal{E}_{\ell,i}(z) := \de_{\eta_{x_\ell}} u_i (z) -  \freq (x_\ell,|z-x_\ell|)\,u_i (z)  \qquad \mbox{for $\ell=1,2$ and $i\in \{1, \ldots Q\}.$}
\] 
By linearity of the (multivalued) differential, we have
\begin{align*}
\de_v u_i (z) = & Du _i (z) \cdot v = Du_i (z) \cdot (z-x_1) - Du_i (z) \cdot (z-x_2) 
= \de_{\eta_{x_1}} u_i (z) - \de_{\eta_{x_2}} u_i (z)\\
= & \underbrace{\big(\freq (x_1,|z-x_1|) - \freq (x_2,|z-x_2|)\big)}_{=: \mathcal{E}_{3} (z)}\, u_i (z) + 
\mathcal{E}_{1,i}(z) - \mathcal{E}_{2,i}(z).
\end{align*}
Substituting the above expression in \eqref{e:derivata} we 
conclude that
\begin{align}
\de_v \freq (x,1) = & \underbrace{2 H_\phi (x,1)^{-1} 
\, \int \sum_i\big(\mathcal{E}_{1,i}  - \mathcal{E}_{2,i} \big)\cdot \de_{\eta_x} u_i\,  d\mu_x}_{=:(A)}
\underbrace{- 2 \frac{D_\phi(x,1)}{H_\phi(x,1)^2} \,\int \sum_i\big(\mathcal{E}_{1,i}  - \mathcal{E}_{2,i} \big) \cdot u_i \, d\mu_x}_{=:(B)} \nonumber\\
&+ \underbrace{2 H_\phi (x,1)^{-1} \left[\int \sum_i \mathcal{E}_3  u_i\de_{\eta_x}u_i d\mu_x  -\freq (x,1)\int \mathcal{E}_3  |u|^2 \, d\mu_x\right]}_{=:(C)}
\label{e:derivata2}
\end{align}
In order to exploit some cancellation property, we re-write $\mathcal{E}_3 (z)$ as
\begin{gather}
 \mathcal{E}_3 (z) = \underbrace{\freq (x_1,1)-\freq (x_2,1)}_{:=\mE} + \underbrace{\freq (x_1,|z-x_1|)-\freq (x_1,1)}_{:=\mE_4(z)}- \underbrace{\qua{\freq (x_2,|z-x_2|)-\freq (x_2,1)}}_{:=\mE_5(z)}\, .
\end{gather}
Note first that $\mu_x$ is supported in $B_1 (x)\setminus B_{1/2} (x)$, thus $\frac{1}{2} \leq |z-x| \leq 1$. Note moreover
that, if $x$ belongs to the segment $[x_1, x_2]$, then $|x-x_\ell|\leq \frac{1}{4}$ and thus we conclude that
$\frac{1}{4} \leq |z-x_\ell|\leq 2$. 

Thus we conclude
\begin{gather}
 \abs{\mE_4(z)}+\abs{\mE_5(z)}\leq W(x_1)+W(x_2) \qquad \forall z\in \supp (\mu_x)\, \forall x\in [x_1, x_2]\, .
\end{gather}
Moreover, notice that 
\begin{align*} 
& \int \mathcal{E} \sum_i u_i \de_{\eta_x} u_i \cdot u_i \, d\mu_x - \freq (x,1) \int \mathcal{E} |u|^2 d\mu_x
= \mathcal{E} \left[\int \sum_i u_i \de_{\eta_x} u_i\cdot u_i \, d\mu_x - D_\phi (x,1)\right]\\
=& \mathcal{E} \left[- \int \phi' (|y-x|) \sum_i \de_{\nu_x} u_i (y) \cdot u_i  (y)\, dy - D_\phi (x,1) \right]\stackrel{\eqref{e:calcolo1}}{=} 0
\end{align*}
This equation is the equivalent of \eqref{eq_f'}, where $\mE$ plays the role of $(d-d')$. Thus we obtain
\begin{align*}
(C) \leq & \qua{W(x_1)+W(x_2)} 2 H_\phi (x,1)^{-1} \int \left[|u|^2 + |u||Du|\right] d\mu_x\\
\leq & \qua{W (x_1) + W (x_2)} 2 H_\phi (x,1)^{-1} \left( 2 H_\phi  (x,1) + \int |Du|^2 d\mu_x \right)\\
\leq & \qua{W (x_1) + W (x_2)} 4 \left(1 + C H_\phi (x,1)^{-1} D_\phi (x,2)\right)\, ,
\end{align*}
where the constant $C$ depends on $\phi$. By \eqref{e:I_spostata} we have $\freq (x,4) \leq C (m, \phi, \Lambda)$ and thus, using \eqref{eq_doubling}
$H_\phi (x,1)^{-1} D_\phi (x,2) \leq C H_\phi (x,2)/ H_\phi (x,1) \leq C$. We have thus concluded $(C) \leq C( W (x_1) + W (x_2))$. 

Coming to (A) observe that, using Cauchy-Schwartz
\begin{align}\label{eq_following}
(A)^2  \leq & 4 H_{\phi} (x,1)^{-2} \int \sum_i |\mathcal{E}_{1,i}   - \mathcal{E}_{2,i} |^2\, d\mu_x \int \sum_i |\partial_{\eta_x} u_i|^2d\mu_x \nonumber\\
\leq & 4 H_\phi (x,1)^{-2} \int \sum_i |\mathcal{E}_{1,i}   - \mathcal{E}_{2,i} |^2\, d\mu_x  \int |Du|^2 d\mu_x\, .
\end{align}
Next, using \eqref{e:I_spostata} we conclude $\freq(x,\tau) \leq C$ for every $\tau\leq 4$ and we can therefore use \eqref{e:H_spostata} and \eqref{eq_doubling} (together with $H_\phi (0,1) =1$) to find a uniform bound from below for $H_\phi (x, \tau)$ when $\tau\in [1/4, 4]$.
Thus, arguing as above we conclude 
\begin{align}
|(A)| \leq C \left(\int \sum_i (|\mathcal{E}_{1,i}|^2 + |\mathcal{E}_{2,i}|^2)\, d\mu_x\right)^{\sfrac{1}{2}}\, .
\end{align}
The same bound is obviously valid for $|(B)|$ as well, following the same arguments. 

Thus in particular we obtain
\begin{align}\label{eq_est_I'}
\de_v \freq (x,1) & \leq  C (W (x_1) + W (x_2)) + C \left(\int \sum_i (|\mathcal{E}_{1,i}|^2 + |\mathcal{E}_{2,i}|^2)\, d\mu_x\right)^{\sfrac{1}{2}}
\end{align}
Let $x_t := tx_1 + 1-t x_2$. 
We next wish to establish the estimate
\begin{equation}\label{e:Fubini3}
 \int \sum_i |\mathcal{E}_{\ell,i}|^2 d\mu_{x_t}  \leq C W (x_\ell) \, ,
\end{equation}
which clearly would complete the proof.
If we introduce the function $\zeta_\ell (y) :=  \sum_i |\mathcal{E}_{\ell,i}|^2 (y)$
we can write 
\begin{align*}
&  \int \sum_i |\mathcal{E}_{\ell,i}|^2 d\mu_{x_t} = \int  \underbrace{ - \phi' \left(|y-x_t|\right) |y- x_t|^{-1}}_{=: m (y)}\zeta_\ell (y)\, dy\, . \nonumber\\
\end{align*}
Observe next that $0 \leq - |y-x_t|^{-1} \phi' (|y-x_t|)\leq 4$ and thus $m (y) \leq 4$. 
\sfava{Recall that $\phi' (s)$ vanishes when $s< \frac{1}{2}$ and $s> 1$. Hence we can assume $\frac{1}{2} \leq |y-x_t|\leq 1$. On the other hand 
$|x_t-x_\ell|\leq \frac{1}{4}$ for every $t\in [0,1]$, hence $\frac{1}{4} \leq |y-x_\ell|\leq \frac{5}{4}$ and so
$m (y) \leq 4 {\bf 1}_{B_2 (x_\ell)\setminus B_{1/4} (x_\ell)} (y)$. Therefore \eqref{e:Fubini3} follows from Proposition \ref{p:pinching}}. We thus conclude the pointwise estimate
\[
\de_v \freq (x,1)  \leq  C (W (x_1) + W (x_2)) \qquad \forall x\in [x_1, x_2]\, .
\]
Indeed reversing the role of $x_1$ and $x_2$ we then conclude
\[
|\de_v \freq (x,1)|  \leq  C (W (x_1) + W (x_2)) \qquad \forall x\in [x_1, x_2]\, .
\]
Integrating the last inequality between any two given points in the segment $[x_1, x_2]$ we derive the desired estimate. 

%%%%%%%%%%%%%%%%%%%%%%%%%%%%%%%%%%%%%%%%%%%%%%%%%%%%%%%%%%%%%%%%%
%
%	SECTION 3
%
%%%%%%%%%%%%%%%%%%%%%%%%%%%%%%%%%%%%%%%%%%%%%%%%%%%%%%%%%%%%%%%%%
\section{\texorpdfstring{$L^2$-Best Approximation}{L2-Best Approximation}}\label{sec_best_app}
Here we prove some distortion bounds in the spirit of \cite{NV}.
We use the standard notation $\dist(y, A) := \inf_{x \in A} |y-x|$.

\begin{definition}\label{d:mean-flat}
Given a Radon measure $\mu$ in $\R^m$ and $k \in \{0, 1, \ldots, m-1\}$,
for every $x \in \R^m$ and for every $r>0$, we define the $k$-th mean flatness of $\mu$ in the ball $B_r (x)$ as 
\begin{equation}\label{e:beta}
D^k_{\mu} (x,r) := \inf_{L} r^{-k-2} \int_{B_r(x)} \dist(y,L)^2\d\mu(y),
\end{equation}
where the infimum is taken among all affine $k$-dimensional planes $L \subset \R^m$.
\end{definition}

\begin{remark} 
 In the literature $D^k_{\mu}$ is  often called the Jones' $\beta_2$ number of dimension $k$ (see for example \cite{davidtoro,AzzTol}). For the aim of this article, we will not need to use any $\beta_p$ for $p\neq 2$, this is why we use this different notation. 
\end{remark}

The following is an elementary characterization of the mean flatness.
Let $x_0 \in \R^m$ and $r_0>0$ be such that $\mu(B_{r_0}(x_0)) >0$, and
let us denote by $\bar x_{x_0,r_0}$ the barycenter of $\mu$ in $B_r(x_0)$, i.e.
\[
\bar x_{x_0,r_0} := \frac{1}{\mu(B_{r_0}(x_0))} \int_{B_{r_0}(x_0)} x \, \d\mu(x)
\]
and let $b:\R^m \times \R^m \to \R$ be the symmetric positive semi-definite bilinear
form given by
\[
b(v, w) := \int_{B_{r_0}(x_0)} \big((x-\bar x_{x_0,r_0}) \cdot v\big)\;\big( (x-\bar x_{x_0,r_0}) \cdot w\big)\,\d\mu(x) \quad\forall \; v,\,w \in \R^m.
\]
By standard linear algebra results there exists
an orthonormal basis of vectors in $\R^m$ that diagonalizes the form $b$:
namely, there is $\{v_1,\ldots,v_m\} \subset \R^m$ (in general not unique) such that
\begin{itemize}
\item[(i)] $\{v_1,\ldots,v_m\}$ is an orthonormal basis: i.e.~$v_i \cdot v_j = \delta_{ij}$;
\item[(ii)] $b(v_i, v_i) = \lambda_i$, for some $0\leq \lambda_m \leq \lambda_{m-1} \leq \cdots \leq \lambda_1$ and $b(v_i, v_j) = 0$ for $i\neq j$.
\end{itemize}

Note that, in particular, by simple manipulations, the following identities hold:
\begin{equation}\label{e:identita chiave}
\int_{B_{r_0}(x_0)} \big((x-\bar x_{x_0,r_0}) \cdot v_i\big)\,x\,\d\mu(x) = \lambda_i\,v_i \quad \forall\; i=1, \ldots, m.
\end{equation}

The $k$-th mean flatness of a measure $\mu$, as well as the optimal planes $L$
in Definition~\ref{d:mean-flat}, can be then characterized in the following way:
let $x_0 \in \R^m$ and $r_0>0$ be such that $\mu(B_{r_0}(x_0)) >0$, then
\begin{equation}\label{e:beta-charac}
D^k_{\mu}(x_0, r_0) = r_0^{-k-2}\sum_{l=k+1}^m\lambda_{l}
\end{equation}
and the infimum in the definition of $D^k_{\mu}$ is reached by all the affine planes $L = x_{x_0,r_0} + \textup{Span}\{v_1, \ldots, v_k\}$
for every choice of an eigenbasis $v_1, \ldots , v_m$ with nonincreasing eigenvalues $\lambda_1\geq \lambda_2 \geq \ldots \geq \lambda_m$. 

The main point of this section is that, if $u$ is as in Assumptions \ref{a:primaria} and \ref{a:secondaria} and $\mu$ is a measure concentrated on the set $\Delta_Q$, its $(m-2)$-th mean flatness is controlled by the pinching $W$. 

\begin{proposition}\label{p:mean-flatness vs freq}
 Under the Assumptions \ref{a:primaria} and \ref{a:secondaria},
there exists $C_{\ref{p:mean-flatness vs freq}} (\Lambda,m,n,Q)>0$
such that the following holds. 
If $\mu$ is a finite nonnegative Radon measure with $\spt(\mu) \subset \Delta_Q$, then
\begin{equation}\label{e:mean-flatness vs freq}
D^{m-2}_{\mu} (x_0,r/8) \leq 
\frac{C_{\ref{p:mean-flatness vs freq}}}{r^{m-2}}
\int_{B_{r/8}(x_0)} \sfava{W^{4r}_{r/8}} (x)\d\mu(x),
\end{equation}
for every $x_0 \in B_{1/8}$ and 
for all $r \in (0, 1]$.
\end{proposition}

The proposition will need the following corollary of Almgren's regularity theory. 

\begin{lemma}\label{l:fava}
Let $\Omega \subset \mathbb R^m$ be a connected open set and $\bar u : \Omega \to \Iqs$ a $\D$-minimizer. Assume there is a
ball $B_{\bar r} (p) \subset \Omega$ and a system of coordinates $x_1, \ldots ,x_m$ for which the restriction of $\bar u$ to
$B_{\bar r} (p)$ is a function of the variable $x_1$ only. Then $\bar u$ is a function of the variable $x_1$ only on 
$\Omega$.
\end{lemma}

\begin{proof}
The lemma is a simple consequence of the unique continuation for harmonic functions when $Q=1$: moreover, it follows easily from the condition $\Delta \alpha =0$ that any harmonic function on a ball $B_{\bar r} (p)$ that depends only on the variable $x_1$ takes the form
$\alpha (x) = a x_1+b$ for some constants $a$ and $b$.  Recalling \cite[Theorem 0.1]{DS0},
there is a (relatively closed) singular set $\Sigma\subset \Omega$ of Hausdorff dimension at most $m-2$ such that, locally
on $\Omega\setminus \Sigma$, the map $\bar u$ is the superposition of $Q$ classical harmonic sheets. Since $\Sigma$ does not disconnect
$\Omega$ we can use the classical theory of harmonic functions to conclude that each such sheet can be written locally as $a x_1+b$ for constants $a$ and $b$. We then easily conclude that $\bar u$ is the superposition of harmonic sheets globally on $\Omega\setminus \Sigma$, each taking the form $a_Q x_1 + b_Q$ for a choice $a_1, \ldots, a_Q$, $b_1, \ldots , b_Q$ of constant vectors in $\mathbb R^n$. This completes the proof. 

\end{proof}

\begin{proof}[Proof of Proposition \ref{p:mean-flatness vs freq}] 
By scale-invariance, we can assume $r=1$ and $H_\phi (0,1)=1$. Without loss of generality we assume that $\mu (B_{1/8})>0$ (otherwise the inequality is obvious) which implies
\begin{equation}\label{e:esiste_Q_punto}
\Delta_Q\cap B_{1/8} \neq \emptyset\, .
\end{equation}

From now on any constant that depends on $\Lambda, m,n$ and $Q$ will be simply denoted by $C$.  Let $\bar x = \bar x_{x_0}$ be the barycenter of $\mu$ in $B_{1/8}(x_0)$,
and let $\{v_1, \ldots, v_m\}$ be any diagonalizing basis for the
bilinear form $b$ introduced above with
eigenvalues $0\leq \lambda_m \leq \lambda_{m-1} \leq \cdots \leq
\lambda_1$.
From \eqref{e:identita chiave} and the definition of barycenter
we also deduce that, for every $j=1, \ldots, m$, for every
$i=1, \ldots, Q$ and for every
$z \in B_{3/2} (x_0) \setminus B_{1/2} (x_0)$, we have
\begin{equation}\label{e:identita chiave2}
-\lambda_j\,v_j \cdot D u_i(z) = 
\int_{B_{1/8} (x_0)} \big((x-\bar x_{}) \cdot v_j\big)\,\big((z-x) \cdot D u_i(z) - \alpha\,u_i(z)\big)\d\mu(x),
\end{equation}
for any constant $\alpha$. In particular the latter identity holds for
\begin{equation}\label{e:alfa}
\alpha:= \frac{1}{\mu(B_{1/8} (x_0))} \int_{B_{1/8} (x_0)} \freq(x, 1) \d\mu(x).
\end{equation}
By squaring the two sides of \eqref{e:identita chiave2}
and summing in $i$ we get
\begin{align*}
\lambda_j^2 \left|\de_{v_j} u(z)\right|^2 & \leq \bigg(\int_{B_{1/8} (x_0)}\sum_i
\big\vert(x-\bar x_{}) \cdot v_j\big\vert\,\big\vert (z-x) \cdot
D u_i(z) - \alpha\,u_i(z)\big\vert\d\mu(x)\bigg)^2\notag\\
& \leq 
\int_{B_{1/8} (x_0)} \sum_i\big((x-\bar x_{}) \cdot v_j\big)^2\d\mu(x) \,
\int_{B_{1/8} (x_0)} \big|(z-x) \cdot D u_i(z) - \alpha\,u_i(z)\big|^2
\d\mu(x)\notag\\
& = \lambda_j\,\int_{B_{1/8} (x_0)}
\sum_i \big|(z-x) \cdot D u_i(z) - \alpha\,u_i(z)\big|^2\d\mu(x)\, ,
\end{align*}
from which we conclude
\begin{equation}\label{e:stima_autovalore}
\lambda_j  \left|\de_{v_j}u (z)\right|^2 \leq \int_{B_{1/8} (x_0)}
\sum_i \big|(z-x) \cdot D u_i(z) - \alpha\,u_i(z)\big|^2\d\mu(x)\, . 
\end{equation}
Integrating with respect $z \in B_{5/4} (x_0)\setminus B_{3/4} (x_0)$
and summing in $j=1, \ldots, m-1$, we finally get
\begin{align}\label{e:intermedio}
& D^{m-2}_{\mu}(x_0,1/8)  \int_{B_{5/4} (x_0)\setminus B_{3/4} (x_0)}
\sum_{j=1}^{m-1}\left|\de_{v_j} u (z)\right|^2\d z\nonumber\\
 = &\int_{B_{5/4} (x_0)\setminus B_{3/4} (x_0)} (\lambda_{m-1} + \lambda_m) \sum_{j=1}^{m-1}
\left|\de_{v_j} u (z)\right|^2\d z \leq 2  \int_{B_{5/4} (x_0)\setminus B_{3/4} (x_0)} \lambda_{m-1} \sum_{j=1}^{m-1}
\left|\de_{v_j} u (z)\right|^2\d z\nonumber\\
&\leq 2 \int_{B_{5/4} (x_0)\setminus B_{3/4} (x_0)} \sum_{j=1}^m \lambda_j\left|\de_{v_j} u (z)\right|^2\d z\notag\\
&\stackrel{\eqref{e:stima_autovalore}}{\leq} C\int_{B_{5/4} (x_0)\setminus B_{3/4} (x_0)}\int_{B_{1/8} (x_0)} 
\sum_i\big|(z-x) \cdot D u_i(z) - 
\alpha\,u_i(z)\big|^2\d\mu(x)\,\d z\notag\\
& \leq C \int_{B_{1/8} (x_0)} \int_{B_{3/2} (x)\setminus B_{1/2} (x)} 
\sum_i \big|(z-x) \cdot D u_i(z) - 
\alpha\,u_i(z)\big|^2\d z\;\d\mu(x).
\end{align}

We next claim that
\begin{equation}\label{e:compattezza_elementare}
\int_{B_{5/4} (x_0)\setminus B_{3/4} (x_0)}
\sum_{j=1}^{m-1}\left|\de_{v_j} u (z)\right|^2\d z \geq c (\Lambda)>0.
\end{equation}
Indeed, since $\freq (0,1) \leq \Lambda$, by \eqref{eq_doubling},  $\int_{B_1} |Du|^2 \leq D_\phi (0, 4)\leq \Lambda H_\phi (0,4)\leq C \Lambda H_\phi (0,1) = C \Lambda$.
If the claim were not correct, there would be a sequence of maps $u^k$ with $\etaa\circ u^k \equiv 0$, 
$u^k (y_0^k) = Q \a{0}$ (recall \eqref{e:esiste_Q_punto}), $\int_{B_2} |Du_k|^2 \leq C \Lambda$, $2 \int_{B_1\setminus B_{1/2}} |u_k|^2 =1$,
but 
\[
\int_{B_{5/4} (x_0^k)\setminus B_{3/4} (x_0^k)}
\sum_{j=1}^{m-1}\left|\de_{v_j} u^k(z)\right|^2\d z \leq \frac{1}{k}\, ,
\]
for some choice of points $x_0^k$, $y^0_k$ in $B_{1/8} (0)$ and of orthonormal vectors $v^k_1, \ldots, v^k_{m-1}$.
By a simple compactness argument, up to extraction of subsequences, $u^k$ would converge to a $\D$-minimizer $\bar u$ such that
$\etaa\circ \bar u \equiv 0$,
\[
\int_{B_1\setminus B_{1/2}} |\bar u|^2 =1 \qquad\mbox{and}\qquad
\int_{B_1} |D\bar u|^2 \leq c\Lambda\, .
\]
Moreover there would be a point $p\in \overline{B}_{1/8}$ and orthonormal vectors $\bar v_1, \ldots, \bar v_{m-1}$ such that
\[
\int_{B_{5/4} (p)\setminus B_{3/4} (p)} \sum_{j=1}^{m-1} \left|\de_{\bar v_j} \bar u\right|^2 = 0\, .
\]
Thus, there is ball $B_\rho (q)\subset B_2 (0)$ over which $\bar u$ is a function of one variable only. By Lemma \ref{l:fava} we conclude that $\bar u$ is a function of one variable on the whole domain $B_2 (0)$. However, since $\bar u (\bar q) = Q \a{0}$for some $\bar q \in B_{1/8}$we conclude that necessarily $\Delta_Q$ has dimension at least $m-1$. However $\bar u$ is nontrivial and thus we would contradict Theorem \ref{t:dim<=m-2}.

\medskip

Next, using \eqref{e:compattezza_elementare} and the triangular inequality in \eqref{e:intermedio} we conclude that
\begin{align*}
D^{m-2}_{\mu} \left(x_0, {\textstyle{\frac{1}{8}}}\right)
&\leq C \underbrace{\int_{B_{1/8} (x_0)} \int_{B_{3/2} (x)\setminus 
B_{1/2} (x)} \sum_i \big|(z-x) \cdot D u_i (z) - \freq (x,1)\,u_i (z)\big|^2\d z\;\d\mu(x)}_{=:(I)}\notag\\
& \quad\quad\quad\quad\quad + C \underbrace{\int_{B_{1/8} (x_0)} \int_{B_{3/2} (x)\setminus 
B_{1/2} (x)}\Big(\freq (x,1) - \alpha 
% \frac{1}{\mu(B_{r})}\int_{B_{r}} I(y, (R+2)r)\d \mu(y)
\Big)^2 \,|u(z)|^2\d z\;\d\mu(x)}_{=: (II)}\, .
\end{align*}
Recalling our choice of $\alpha$ in \eqref{e:alfa} we can estimate 
the second integral easily as
\begin{align}
(II) &\leq C  \int_{B_{1/8} (x_0)} \Big(I(x,1) - \frac{1}{\mu(B_{1/8} (x_0))}\int_{B_{1/8} (x_0)} \freq (y,1)\d \mu(y)\Big)^2 \d\mu(x)\notag\\
& = C  \int_{B_{1/8} (x_0)} \left(\frac{1}{\mu(B_{1/8}(x_0))}\int_{B_{1/8} (x_0)}
\big( \freq (x,1) -  \freq (y, 1)\big) \d \mu(y)\right)^2\;\d\mu(x)\notag\\
&\leq \frac{C}{\mu(B_{1/8}(x_0))}  \int_{B_{1/8} (x_0)} \int_{B_{1/8} (x_0)}
\big( \freq (x,1) -  \freq (y, 1)\big)^2 \d \mu(y)\;\d\mu(x)\nonumber
\end{align}
Thus, using Theorem \ref{t:stima_fondamentale} we conclude
\begin{align}
\sfava{(II)} &\leq  \frac{C}{\mu(B_{1/8}(x_0))} \int_{B_{1/8} (x_0)} \int_{B_{1/8} (x_0)}
\sfava{\big(W^{4}_{1/8}(x) + W^{4}_{1/8}(y)}\big)
\d \mu(y)\;\d\mu(x)\nonumber\\ &= 2C \int_{B_{1/8} (x_0)} \sfava{W^4_{1/8} (x)} \d \mu(x).\label{e:(I)}
\end{align}
As for the first integral, we split it as
\begin{align}
& (I) \leq   C \underbrace{\int_{B_{1/8} (x_0)}  \int_{B_{3/2} (x)\setminus B_{1/2}(x)} (\freq (x, 1) - \freq (x, |z-x|))^2 |u|^2
\d z\;\d\mu(x)}_{=:(I_1)}\notag\\
& + C \underbrace{\int_{B_{1/8} (x_0)}  \int_{B_{3/2} (x)\setminus B_{1/2}(x)} 
\sum_i \big|(z-x) \cdot D u_i(z) - 
\freq  (x, |z-x|) \,u_i(z)\big|^2\d z\;\d\mu(x)}_{(I_2)}.
\end{align}
Observe now that, for $z$ in the domain of integration, and $x\in \spt(\mu) \cap \B {1/8}{0}$, $1/4\leq |z-x|\leq 4$ and thus, by the monotonicity of the frequency function,
\[
|\freq (x, |z-x|) - \freq (x,1)| \leq \freq (x, 4) - \freq (x,1/4)\, ,
\]
which leads to
\begin{align}
(I_1) \leq& C H_\phi (0,1) \int_{B_{1/8} (x_0)} W^4_{1/4} (x)^2\, d\mu (x)
 \leq C
 \int_{B_{1/8} (x_0)} W^4_{1/4} (x)\, d\mu (x)\, .\label{e:I_1}
\end{align}
As for $(I_2)$ by Proposition \ref{p:pinching}
\begin{align}
&   \int_{B_{3/2} (x)\setminus B_{1/2}(x)} 
\sum_i \big|(z-x) \cdot D u_i(z) - 
\freq (x, |z-x|) \,u_i(z)\big|^2\d z \leq  C\sfava{W^4_{1/8}} (x) \, .
\end{align}
Integrating the latter inequality in $x$ and adding the estimate \eqref{e:I_1} we conclude
\begin{align}
(I) \leq & C \int_{B_{1/8}(x_0)} \sfava{W^4_{1/8}} (x)\, d\mu (x)\, .\label{e:(II)}
\end{align}
The inequalities \eqref{e:(I)} and \eqref{e:(II)} clearly complete the proof of \eqref{e:mean-flatness vs freq}.
\end{proof}

\section{Approximate spines}

It is well known that for $Q$-valued functions of dimension $2$ $\Delta_Q$ is discrete, see for example \cite[corollary 3.4]{GhiSpo}. This is a consequence of the fact that if $\freq (0,2)-\freq (0,1/4)$ is sufficiently small, then $\Delta_Q \cap \ton{\B {1}0 \setminus \B {1/2}{0}}=\emptyset$. For functions of $m$ variables, a similar statement is true if we assume pinching of the frequency over $m-1$ points that are sufficiently spread. In this section, if $A$ is a subset of $\mathbb R^N$, we denote by $\spanna\, A$ the linear subspace generated by the elements of $A$ (with the usual convention $\spanna \, \emptyset = \{0\}$). 

\begin{definition}
 Given a set of points $\cur{x_i}_{i=0}^k\subset \B {r}{x}$, we say that this set of points are $\rho r$-linearly independent if for all $i=1,\cdots, k$:
 \begin{gather}
  d(x_i, x_0+\spanna\, \{x_{i-1}-x_0,\cdots,x_1-x_0\})\geq \rho r\, .
 \end{gather}
\end{definition}

\begin{definition}
 Given a set $F\subset B_r (x)$, we say that $F$ $r\rho$-spans a $k$-dimensional affine subspace $V$ if there exists $\cur{x_i}_{i=0}^k\subseteq F$ that are $r\rho$-linearly independent and $V = x_0 + \spanna\, \{x_i-x_0\}$.
\end{definition}

The following simple geometric remark will play an important role in the next section: 

\begin{remark}\label{r:m-3}
If a set $F\cap B_r (x)$ does not $r\rho$-span a $k$-dimensional affine subspace, then it is contained in $B_{\rho r} (L)$ for some $(m-3)$-dimensional subspace $L$. The proof is very easy, but we include it for the reader's convenience. First of all, by scaling we can assume that $r=1$. Now pick the maximal $\kappa\in \mathbb N$ for which there is a set $\{x_0, \ldots, x_\kappa\}\subset F$ that $\rho$-spans a $\kappa$-dimensional affine space $L$.
Clearly we must have $\kappa < k$ but also $F\subset B_\rho (L)$: the latter is given by the maximality of $\kappa$ because if there were $y\in F\setminus B_\rho (L)$, then $\{x_0, \ldots, x_\kappa, y\}$ would $\rho$-span a $(\kappa+1)$-dimensional space. 
\end{remark}

\begin{lemma}\label{lemma_n-2_pinch}
 Let $u$ be as in Assumptions \ref{a:primaria} and \ref{a:secondaria}. Let \sfava{$\rho, \bar\rho, \tilde{\rho}\in ]0,1[$} be given. There exists an $\epsilon=\epsilon(m,n,Q,\Lambda, \rho, \bar \rho, \tilde \rho)>0$ such that the following holds.
 
If $\cur{x_i}_{i=0}^{m-2}\subset B_1 (0)$ is a set of $\rho$-linearly independent points such that
 \begin{gather}
W^2_{\tilde \rho} (x_i) =   \freq (x_i,2)-\freq (x_i,\tilde \rho)<\epsilon \qquad \forall i\, ,
 \end{gather}
then 
\begin{gather}
 \Delta_Q \cap \ton{\B 1 0 \setminus \B {\bar \rho} {V}} = \emptyset\, ,
\end{gather}
where $V = x_0 + \spanna\, \cur{x_i -x_0: 1\leq i \leq m-2}$.
\end{lemma}

Under the same assumptions of the previous lemma, we also obtain that $\freq (x,r)$ is almost constant on $V$ if $r$ is not much smaller than $\tilde \rho$.
\sfava{In fact, a suitable modification of the proof of Theorem \ref{t:stima_fondamentale} leads to the following much more precise estimate when we estimate the oscillation of the frequency function at the same scale. Since, however, such a precise control is not needed later, we omit its proof.}

\begin{proposition}\label{cor_N_span}
 Fix any $\rho>0$, and consider the set
 \begin{gather}
  F(\delta) = \cur{y\in B_{1/8} (0) \ \ s.t. \ \ \sfava{W^4_{1/8} (y)}\leq \delta}\, .
 \end{gather}
If $F$ $\rho/8$-spans some subspace $V$, then for all $y,y'\in V\cap B_{1/\sfava{32}} (0)$
\begin{gather}\label{eq_diff_I_on_V}
  \abs{\freq (y,1)-\freq (y',1)}\leq C\sqrt{\delta}\, ,
\end{gather}
where $C=C(\Lambda,m,n,Q,\rho)$.
\end{proposition}

Indeed we need a less precise version of such oscillation bound at all scales between $\tilde{\rho}$ and $1$. We record the precise statement in the following lemma \sfava{for which we provide a proof later}. 

\begin{lemma}\label{lemma_N_span}
 Let $u$ be as in Assumptions \ref{a:primaria} and \ref{a:secondaria} and \sfava{$\rho, \tilde{\rho}, \bar \rho\in ]0,1[$} be given. For all $\delta>0$, there exists an $\epsilon=\epsilon(m,n,Q,\Lambda,\rho, \tilde\rho, \bar \rho, \delta)>0$ such that the following holds.
 
 Let $\cur{x_i}_{i=0}^{m-2}\subset \B 1 0$ be a set of $\rho$-linearly independent points, and assume that for all $i$:
 \begin{gather}
 W^2_{\tilde \rho} (x_i) =  \freq (x_i,2)-\freq (x_i,\tilde \rho)<\epsilon\, \, .
 \end{gather}
Then for all $y, y'\in \B 1 0 \cap V$ and for all $r, r'\in [\bar \rho, 1]$ we have
\begin{gather}
 \abs{\freq (y,r)-\freq (y',r')}\leq \delta\,  .
\end{gather}
where $V = x_0 + \spanna\, \cur{x_i-x_0: 1 \leq i \leq m-2}$.
\end{lemma}

\subsection{Compactness and homogeneity} The rest of the section is devoted to proving the above lemmas. In both cases we will argue by compactness. The crucial ingredients are
the following proposition, where we show that a uniform control upon the frequency function $\freq$ ensures strong $L^2$ compactness, and the subsequent elementary lemma.

\begin{proposition}\label{p:compattezza}
Let $u_q: B_r (x) \to \Iqs$ be a sequence of $W^{1,2}$ maps minimizing the Dirichlet energy with the property that 
\[
\sup_q \left(I_{\phi,u_q} (x,r) + H_{\phi, u_q} (x,r)\right) < \infty\, .
\]
Then, up to subsequences, $u_q$ converges strongly in $L^2$ to a map $u\in W^{1,2}_{loc}$. Moreover $u$ is a local minimizer, namely its restriction to any open set $\Omega\subset\subset B_r (x)$ is a minimizer, and the convergence is locally uniform and strong in $W_{loc}^{1,2}$.
\end{proposition}

\begin{lemma}\label{l:invarianza}
Let $u: \mathbb R^m \to \Iqs$ be a continuous map that is radially homogeneous with respect to two points $x_1$ and $x_2$, namely 
there exists positive constants $\alpha_1$ and $\alpha_2$ such that
\begin{align*}
u (x) = & \sum_i \a{|x-x_1|^{\alpha_1} u_i \left({\textstyle{\frac{x-x_1}{|x-x_1|}}}+x_1\right)} \qquad \forall x\neq x_1\\ 
u (x) = & \sum_i \a{|x-x_2|^{\alpha_2} u_i \left({\textstyle{\frac{x-x_2}{|x-x_2|}}}+x_2\right)} \qquad \forall x\neq x_2\, .
\end{align*}
Then $\alpha_1=\alpha_2$, $u$ is invariant along the $x_2-x_1$ direction, namely $u (y + \lambda (x_2-x_1)) = u (y)$ for every $y$ and every $\lambda\in \mathbb R$, and finally $u (\lambda x_1+  (1-\lambda) x_2) = Q \a{0}$ for every $\lambda\in \mathbb R$. 
\end{lemma}

A last technical observation which will prove useful here and in other contexts is the following ``unique continuation'' type 
result for $Q$-valued minimizers of the Dirichlet energy.

\begin{lemma}\label{l:uniq_cont}
Let $\Omega \subset \mathbb R^m$ be a connected open set and $u,v : \Omega \to \Iqs$ two maps with the following property:
\begin{itemize}
\item both $u$ and $v$ are local minimizers of the Dirichlet energy, namely for every $p\in \Omega$ there exists a neighborhood $U$ such that $u|_U$ and $v|_U$ are both minimizers;
\item $u$ and $v$ coincide on a nonempty open subset of $\Omega$.
\end{itemize}
Then $u$ and $v$ are the same map.
\end{lemma}

\begin{proof}[Proof of Proposition \ref{p:compattezza}] After suitable scaling, translation and renormalization we can assume that $B_r (x) = B_1 (0)$ and that $H_{\phi, u_q} (0,1) =1$. We therefore conclude that $D_{\phi, u_q} (0,1) $ is uniformly bounded and that $Du_q$ is uniformly bounded in $L^2 (B_\rho)$ for every $\rho<1$, because
\[
\int_{B_\rho (0)} |Du_q|^2 \leq \frac{1}{2-2\rho} D_{\phi, u_q} (0,1) \qquad \forall \rho\in ]{\textstyle{\frac{1}{2}}},1[\, . 
\]
Observe also that 
\[
\int_{B_1 (0)\setminus B_{1/2} (0)} |u_q|^2 \leq H_{\phi, u_q} (0,1)\, ,
\]
which combined with the uniform control of $\int_{B_{2/3} (0)} |Du_q|^2$ gives a uniform estimate on $\int_{B_1 (0)} |u_q|^2$.
Hence the sequence $(u_q)$ is uniformly bounded in $W^{1,2} (B_\rho (0))$ for every $\rho<1$: the compact embedding of $W^{1,2} (B_\rho (0))$ in $L^2 (B_\rho (0))$ (cf. \cite[Proposition 2.11]{DS0}) and a standard diagonal argument gives the existence of a subsequence, not relabeled, converging strongly in $L^2_{loc}$ to a $W^{1,2}_{loc}$ map $u$. 

We claim next the existence of a constant $C$ such that
\begin{equation}\label{e:H_stima_uniforme}
H_{u_q} (0, \rho) = \int_{\partial B_\rho} |u_q|^2 \leq C \qquad \forall q \;\mbox{and}\; \forall \rho\in ]\frac{1}{2}, 1[\, .
\end{equation}
The latter clearly implies that
\[
\int_{B_1 (0)\setminus B_\rho (0)} |u_q|^2 \leq C (1-\rho)
\]
and thus upgrades the strong $L^2_{loc}$ convergence to strong convergence in $L^2 (B_1 (0))$. Arguing as in \cite[Proof of Theorem 3.15]{DS0} we derive that the map $\rho \mapsto h_q (\rho) = H_{\phi, u_q} (0, \rho)$ belongs to $W^{1,1}_{loc}$ and we compute
\[
h_q' (\rho) = \frac{m-1}{r} \int_{\partial B_\rho} |u_q|^2 + 2 \int_{B_\rho (0)} |Du_q|^2
\]
(cf. \cite[(3.46)]{DS0}. 
Integrating in $\rho$ we then conclude
\[
\int_{1/2}^1 |h'_q (\rho)|\, d\rho \leq C \int_{B_1 (0)\setminus B_{1/2} (0)} |u_q|^2 + \underbrace{2\int_{1/2}^1 \int_{B_\rho (0)} |Du_q (x)|^2 dx\,  d\rho}_{(I)}\, .
\]
On the other hand notice that reversing the order of integration in (I) we easily conclude
\begin{align*}
(I) = & \int |Du_q (x)|^2 \phi (|x|)\, dx = D_{\phi, u_q} (0,1)\, .
\end{align*}
Hence the sequence $h_q$ is uniformly bounded in $W^{1,1} (]\frac{1}{2},1[)$, which in turn gives a uniform bound on its $L^\infty$ norm. This completes the proof of the first part of the proposition. The local uniform convergence follows instead from \cite[Theorem 3.19]{DS0}, whereas the local minimality of $u$ and its strong convergence in $W^{1,2}_{loc}$ follows from \cite[Proposition 3.20]{DS0}. 
\end{proof}

\begin{proof}[Proof of Lemma \ref{l:invarianza}]
We start by observing that $u (x_1) = u (x_2) = Q \a{0}$ simply by homogeneity and continuity. Moreover, if we show the invariance of the function along the $x_2-x_1$ direction, then the equality $\alpha_1=\alpha_2$ is a triviality. After translating and rescaling we can assume, without loss of generality, that $x_1=0$ and that $x_2= e = (1,0,0,\ldots , 0)$. We let $(z_1, \ldots , z_m)$ be the corresponding standard Cartesian coordinates on $\mathbb R^m$. Our goal is to show that $u$ is a function of the  variables $z' = (z_2, \ldots, z_m)$ only. 

We first claim that
\begin{equation}\label{e:trasla_1}
u (e+w) = u (w)\, .
\end{equation}
The identity is obvious if $w=0$. Fix thus $w\neq 0$. 
\[
\lambda w = e + |\lambda w -e| \frac{\lambda w -e}{|\lambda w -e|} =: e+ |\lambda w-e|\, w_\lambda\, .
\]
Note that for $\lambda\to \infty$, $e+w_\lambda \to e + \frac{w}{|w|}$.
Using the homogeneity of the function we then conclude
\begin{equation}\label{e:lambda_to_inf}
\sum_i \a{\lambda^{\alpha_1} u_i (w)} = \sum_i \a{|\lambda w-e|^{\alpha_2} u_i (e+w_\lambda)}\, .
\end{equation}
Clearly, if $u (w) = Q\a{0}$, then $u (e+w_\lambda)=Q \a{0}$ and sending $\lambda$ to infinity we conclude
$u (e+\frac{w}{|w|})= Q \a{0}$: thus by homogeneity $u (e+w) = Q \a{0} = u (w)$. With a symmetric argument we conclude that if $u (e+w) = Q \a{0}$, then $u(w) = Q \a{0}= u (e+w)$. If both $u (w)$ and $u (e+w)$ are different from $Q \a{0}$, then sending $\lambda \to \infty$ we conclude that the
\[
\lim_{\lambda\to \infty} \frac{|e-\lambda w|^{\alpha_2}}{\lambda^{\alpha_1}}
\]
exists, it is finite and nonzero. Hence $\alpha_1=\alpha_2$, which implies that the limit is indeed $|w|$. Plugging this information in \eqref{e:lambda_to_inf}, sending $\lambda$ to infinity and using the homogeneity of $u$ we achieve \eqref{e:trasla_1}. 

Next consider $z_1>0$ and $z'\in \mathbb R^{m-1}$. We then have
\[
u (z_1, z') = \sum_i \a{z_1^{\alpha_1} u_i (1, z_1^{-1} z')} = \sum_i \a{z_1^{\alpha_1} u_i (0, z_1^{-1} z')} = u (0, z')\, .
\]
If instead $z_1<0$, we can then argue
\[
u (z_1, z') = \sum_i \a{(-z_1)^{\alpha_1} u_i (-1, (-z_1)^{-1} z')} = \sum_i \a{(-z_1)^{\alpha_1} u_i (0, (-z_1)^{-1} z')} = u (0,z')\, . \qedhere
\]
\end{proof}

\begin{proof}[Proof of Lemma \ref{l:uniq_cont}]
We prove it by induction over $Q$. For $Q=1$ the statement is the unique continuation for classical harmonic functions. Assume therefore that $Q_0>1$ and that the claim has been proved for every $Q<Q_0$. Let $\Delta_Q (u)$ be the set of points where $u = Q \a{\etaa\circ u}$. We know from \cite[Proposition 3.22]{DS0} that, either $\Delta_Q (u)$ coincides with $\Omega$, or it has dimension at most $m-2$. If it coincides with $\Omega$, then $\Delta_Q (v)$ has nonempty interior and again invoking \cite[Proposition 3.22]{DS0} we conclude that $\Delta_Q (v)= \Omega$. In this case $v = Q \a{\etaa\circ v}$ and $u = Q \a{\etaa\circ u}$: since $\etaa\circ u$ and $\etaa\circ v$ are harmonic functions that coincide on a nonempty open set,
they coincide over all $\Omega$ and we conclude $u=v$. 

We can thus assume that both $\Delta_Q (u)$ and $\Delta_Q (v)$ have dimension at most $m-2$. Therefore the open set $\Omega' :=
\Omega \setminus (\Delta_Q (u) \cup \Delta_Q (v))$ is a connected open set. Clearly, by continuity of $u$ and $v$ it suffices to show that $u$ and $v$ coincide on $\Omega'$. Consider therefore in $\Omega'$ the set $\Gamma$ which is the closure of the interior of $\{u=v\}$. Such set is nonempty and closed. If we can show that it is open  the connectedness of $\Omega'$ implies $\Gamma = \Omega'$. 

Let thus $p$ be a point in $\Gamma$. Clearly there are $T\in \I{Q_1} (\mathbb R^n)$ and $S\in \I{Q_2} (\mathbb R^n)$ with $Q_1+Q_2 =Q$, $\supp (T)\cap \supp (S) =\emptyset$ and $u (p) = v (p) = T+S$. In particular, there is a $\delta>0$ such that
$\max \{\cG (T', T), \cG (S,S')\} \leq \delta$ implies $\supp (T')\cap \supp (S') = \emptyset$. It follows that any $Q$-point $P$ with
$\cG (P, T+S) < \delta$ can be decomposed in a unique way as $S'+T'$ with $\cG (S', S), \cG (T', T) < \delta$. 

Using the continuity of $u$ and $v$, in a sufficiently small ball $B_\rho (p)$ we have 
\[
\| \cG (u, T+S)\| + \|\cG (v, T+S)\|<\delta\, .
\] 
In particular this defines in a unique way continuous maps $u_1, u_2, v_1, v_2$ such that $u|_{B_\rho (p)} = u_1+ u_2$, $v|_{B_\rho (p)} = v_1+v_2$ and 
\[
\|\cG (u_1, T)\|_0, \|\cG (u_2, S)\|_0, \|\cG (v_1, T)\|_0, \|\cG (v_2, T)\|_0 < \delta\, .
\] 
Note moreover that, by possibly choosing $\rho$ smaller, we can assume that both $u|_{B_\rho (p)}$ and $v|_{B_\rho (p)}$ are minimizers. It follows then that the maps $u_i$ and $v_i$ must be minimizers of the Dirichlet energy.
By definition of $\Gamma$, there is a nonempty open set $A\subset B_\rho (p)$ where $u$ and $v$ coincide. Given the uniqueness of the decomposition $P = S'+T'$ discussed above when $\cG (P, T+S) < \delta$, we conclude that $u_1=v_1$ and $u_2=v_2$ on $A$. By inductive assumption, this implies that $u_1=v_1$ and $u_2=v_2$ on the whole ball $B_\rho (p)$. In other words $B_\rho (p) \subset \Gamma$ and thus $p$ is an interior point of $\Gamma$. By the arbitrariness of $p$ we conclude that $\Gamma$ is open, thus completing the proof. 
\end{proof}

\subsection{Proof of Lemma \ref{lemma_n-2_pinch}} Assume by contradiction that the lemma does not hold. Then there is a sequence of $u_q$ satisfying the Assumptions \ref{a:primaria} and \ref{a:secondaria} and a sequence of collections of points $P_q = \{x_{q,0}, x_{q,1}, \ldots, x_{q, m-2}\}$ with the following properties:
\begin{itemize}
\item each $P_q$ is $\rho$-linearly independent for some fixed $\rho>0$;
\item $I_{\phi, u_q} (x_{q,i}, 2) - I_{\phi, u_q} (x_{q,i}, \tilde{\rho}) \to 0$ as $q\to \infty$ for some fixed $\tilde{\rho} >0$;
\item $\Delta_Q (u_q)\cap (B_1 (0)\setminus B_{\bar \rho} (V_q))$ contains at least one point $y_q$, where $\bar \rho >0$ is some fixed constant and $V_q = x_{q,0} + \spanna\, \{ x_{q,1} - x_{q,0}, \ldots, x_{q,m-2} - x_{q,0}\}$). 
\end{itemize}
Without loss of generality we can assume that $H_{\phi, u_q} (0, 64) =1$. Recalling that $I_{\phi, u_q} (0, 64) \leq \Lambda$, we can apply the Proposition \ref{p:compattezza} and, up to a subsequence not relabeled, assume that
\begin{itemize}
\item $u_q\to u$ in $L^2 (B_{64} (0))$ and locally uniformly;
\item $u$ is a minimizer of the Dirichlet energy and $u_q\to u$ strongly in $W^{1,2}_{loc}$;
\item $P_q$ converges to some $\rho$-linearly independent set $P = \{x_0, \ldots, x_q\}$;
\item the points $y_q$ converge to some $y\in \bar{B}_1 (0)$ with $u (y) = Q \a{0}$. 
\end{itemize}
Observe first that $H_{\phi, u} (0, 64) = 1$ and that $\etaa\circ u \equiv 0$. By \cite[Proposition 3.22]{DS0}, either $\Delta_Q (u)$ has Hausdorff dimension at most $m-2$, or $u = Q \a{\zeta}$ for some classical harmonic function $\zeta$. The latter alternative would however imply $\zeta = \etaa \circ u \equiv 0$ and hence $H_{\phi, u} (0, 64) = 0$. We conclude therefore that $\Delta_Q (u)$ has dimension at most $m-2$. 

In particular $H_{\phi, u} (x, \rho) \neq 0$ for any positive $\rho$. In turn we conclude from the convergence properties of $u_q$ that $I_{\phi, u_q} (y_q, \rho) \to I_{\phi, u} (y, \rho)$ whenever $\rho< 64 - |y|$ and $y_q \to y$. Hence we infer that
\[
I_{\phi, u} (x_i, 2) = I_{\phi, u} (x_i, \tilde{\rho})\, .
\]
In turn this implies that the function $u$ is homogeneous in $|x-x_i|$ in the annulus $B_2 (x_i)\setminus B_{\tilde{\rho}} (x_i)$ with homogeneity exponent $\alpha_i\geq 0$. We can thus extend $u$ to a function $v_i$ with the same homogeneity over the whole $\mathbb R^m$. A simple rescaling argument implies that for every $p\neq 0$ there is a neighborhood $U$ of $p$ where $v_i$ is a minimizer of the Dirichlet energy. Using Lemma \ref{l:uniq_cont}, $v_i$ and $u$ coincide on $B_{64} (0)\setminus \{x_i\}$. But then by continuity we conclude that $u=v_i$ on $B_{64} (0)$. 

Hence we have that
\begin{equation}\label{e:omogenee}
u (x) = \sum_j \a{ |x-x_i|^{\alpha_i} u_j \left(x_i + {\textstyle{\frac{x-x_i}{|x-x_i|}}}\right)}\, .
\end{equation}
Note that, if $\alpha_i$ were $0$, then the map $u$ would take a constant value different from $Q\a{0}$, which is not possible because $u (y) = Q \a{0}$. Thus each $\alpha_i$ is positive. 

Now, although $u$ is defined on $B_{64} (0)$, using its homogeneity with respect to any of the points $x_i$, it could be extended 
to a map $v_i$ on the whole $\mathbb R^m$, as done above. Each such extension would be a local minimizer of the Dirichlet energy and, by unique continuation
(cf. Lemma \ref{l:uniq_cont}), all such extensions must coincide. We can therefore consider $u$ as defined on the whole space $\mathbb R^m$, with \eqref{e:omogenee} valid everywhere and for every $x_i$.
Using Lemma \ref{l:invarianza} we conclude that, if $V = x_0 + \spanna\, \{x_i-x_0: 1\leq i \leq m-2\} =: x_0 + V$, then $u$ is a function of the variables orthogonal to $V$ and $u (x_0+v) = Q \a{0}$ for every $v\in V$. On the other hand, since the notion of $\rho$-linear independence is stable under convergence, $V$ is an $(m-2)$-dimensional space. Lemma \ref{l:invarianza} implies also that the $\alpha_i$'s are equal to a number $\alpha$. Summarizing, if we denote by $S$ the unit circle of the two dimensional space $V^\perp$, we have that there is a continuous map $\zeta:S \to \Iqs$ such that 
\begin{equation}\label{e:spinone}
u (x_0 + v + \lambda w) = \sum_j \a{\lambda^\alpha \zeta_j (w)} \qquad \forall v\in V, \forall w\in S, \forall \lambda \geq 0\, . 
\end{equation}
On the other hand the point $y$ (which is the limit of the points $y_q$) cannot belong to $V$. Since $u(y) = Q \a{0}$, we would conclude that $u\equiv Q\a{0}$ on the $(m-1)$-dimensional space $x_0 + \spanna\, L\cup \{y-x_0\}$. 
This however is a contradiction with the dimension estimate on $\Delta_Q (u)$. 

\subsection{Proof of Lemma \ref{lemma_N_span}} The proof is entirely analogous to the previous one. Again by contradiction assume that the statement is false. Then there is a sequence of $u_q$ satisfying the Assumptions \ref{a:primaria} and \ref{a:secondaria} and a sequence of collections of points $P_q = \{x_{q,0}, x_{q,1}, \ldots, x_{q, m-2}\}$ with the following properties:
\begin{itemize}
\item each $P_q$ is $\rho$-linearly independent for some fixed $\rho>0$;
\item $I_{\phi, u_q} (x_{q,i}, 2) - I_{\phi, u_q} (x_{q,i}, \tilde{\rho}) \to 0$ as $q\to \infty$ for some fixed $\tilde{\rho} >0$;
\item if $V_q = x_{q,0} + \spanna\, \{ x_{q,1} - x_{q,0}, \ldots, x_{q,m-2} - x_{q,0}\}$, then there are two points $y_{q,1}, y_{q,2} \in (x_{q,0} + V_q)\cap B_1 (0)$ and two radii $r_{q,1}, r_{q,2}\in [\bar \rho, 1]$ with the property that
\begin{equation}\label{e:non_vicini}
|I_{\phi, u_q} (y_{q,1}, r_{q,1}) - I_{\phi, u_q} (y_{q_2}, r_{q,2})|\geq \delta>0\, .
\end{equation}
\end{itemize}
Without loss of generality we can assume that $H_{\phi, u_q} (0, 64) =1$. Recalling that $I_{\phi, u_q} (0, 64) \leq \Lambda$, we can apply the Proposition \ref{p:compattezza} and, up to a subsequence not relabeled, assume that
\begin{itemize}
\item $u_q\to u$ in $L^2 (B_{64} (0))$ and locally uniformly;
\item $u$ is a minimizer of the Dirichlet energy and $u_q\to u$ strongly in $W^{1,2}_{loc}$;
\item $P_q$ converges to some $\rho$-linearly independent set $P = \{x_0, \ldots, x_q\}$;
\item the points $y_{q, i}$ converge to some $y_i$ and the radii $r_{q,i}$ to some $r_i\in [\bar\rho, 1]$. 
\end{itemize}
Again arguing as above the plane $L = x_0+\spanna \{x_i-x_0 : 1\leq i \leq m-2\} = x_0+V$ is $(m-2)$-dimensional and 
$u$ has the form \eqref{e:spinone} for some $\alpha >0$. We conclude that 
\begin{equation}\label{e:uguali}
I_{\phi, u} (x, r) = \alpha \qquad \mbox{for any $r>0$ and any $x\in L$.}
\end{equation}
On the other hand $y_1, y_2\in L$ and $I_{\phi, u_q} (y_{q,i}, r_{q,i}) \to I_{\phi, u} (y_i, r_i)$. Thus \eqref{e:non_vicini} and
\eqref{e:uguali} are in contradiction. 

\section{Minkowski-type estimate}

In this section we combine the previous theorems with the Reifenberg-type methods developed in \cite{NV} to give a proof of the Minkowski upper bound in Theorem \ref{t:Minkio}. We follow in particular the simplified construction of \cite{NVapp}.

The following result, which we simply quote from \cite[Theorem 3.4]{NV}, allows us to turn a small bound on the mean flatness into volume bounds for a general measure $\mu$. Note that generalizations of this result appeared recently in \cite{Mis,ENV}.

\begin{theorem}[{\cite[Theorem 3.4]{NV}}]\label{th_reif_vol}
Fix $k\leq m \in \N$, let $\{\B {s_j}{x_j} \}_{j\in J}\subseteq \B 2 0\subset \R^m$ be a sequence of pairwise disjoint balls centered in $\B 1 0$, and let $\mu$ be the measure
\begin{gather}
 \mu = \sum_{j\in J } s_j^k \delta_{x_j}\, .
\end{gather}
There exist constants $\delta_0 = \delta_0(m)$ and $C_R = C_R(m)$ depending only on $m$ such that if for all $\B r x \subseteq \B 2 0$ with $x\in \B 1 0$ we have the integral bound
\begin{gather}\label{est_D_reif}
 \int_{\B {r} x } \ton{\int_0^r D^k_\mu(y,s)\,{\frac{ds}{s}}}\, d\mu(y)<\delta_0^2 r^{k}\, ,
\end{gather}
then the measure $\mu$ is bounded by
\begin{align}
\mu(\B 1 0) = \sum_{j\in J } s_j^k\leq C_R\, .
\end{align}
\end{theorem}

\subsection{Efficient covering} In fact the latter theorem and the results of the previous sections will be used to prove the following intermediate step

\begin{proposition}\label{lemma_cover_final} Let $u$ be as in the Assumptions \ref{a:primaria} and \ref{a:secondaria}. Fix any $x\in B_{1/8} (0)$ and $0<s<r\leq 1/8$. Let $D\subseteq \Delta_Q\cap \B r x$ by any subset of $\Delta_Q$, and set $U = \sup\cur{\freq (y,r) \ \ \vert \ \ y\in D}$. There exist a positive $\delta = \delta_{\ref{lemma_cover_final}}=\delta(m,n,Q,\Lambda)$, a constant $C_V = C_V(m)\geq 1$, a finite covering with balls $B_{s_i} (x_i)$
and a corresponding decomposition of $D$ in sets $A_i\subset D$ with the following properties:
\begin{itemize}
\item[(a)] $A_i \subset B_{s_i} (x_i)$ and $s_i \geq s$;
\item[(b)] $\sum_i s_i^{m-2}\leq C_V r^{m-2}$;
\item[(c)] for each $i$, either $s_i=s$, or 
\begin{gather}\label{eq_Edrop}
\sup \{ \freq\, (y,s_i): y \in A_i \} \leq U- \delta\, .
\end{gather}
\end{itemize}
\end{proposition}

With this proposition at hand the theorem follows easily

\begin{proof}[Proof of Theorem \ref{t:Minkio}]
We consider the set $D_0 := \Delta_Q\cap B_1 (0)$ and recall that, by Lemma \ref{l:limiti uniformi_II},
\begin{equation}\label{e:start}
U_0 = \sup \{ \freq (y, 1/8): y\in D_0\} \leq C (\Lambda +1)\, .
\end{equation}
Apply Proposition \ref{lemma_cover_final} with $r=1$, $s=\rho$ and $D=D_0$ and let $\{A_i\}$ and $\{B_{s_i} (x_i)\}$, $i\in I_1$, be the corresponding decomposition and covering of $D_0$. In particular
\[
\sum_{i\in I_1} s_i^{m-2} \leq C_V\, .
\]
Let $I_1^g :=\{i : s_i = \rho\}$. For each $s_i > \rho$ we instead have the frequency drop
\[
\sup \{ \freq\, (y,s_i): y \in A_i \} \leq U_0 - \delta\, .
\]
For every $i\in I_1\setminus I_1^g$ apply the Proposition \ref{lemma_cover_final} again with $D = A_i$, $r=s_i$ and $s=\rho$. We then find a decomposition $\{A_{i,j}\}$ of each $A_i$ and corresponding balls $\{B_{s_{i,j}} (x_{i,j})\}$, $j\in I_1^i$, with 
\[
\sum_{j\in I_1^i} s_{i,j}^{m-2} \leq C_V s_i^{m-2}\, .
\]
We now define $I_2$ as the union of $I_1^g$ and all $I_j^i$ with $i\not\in I_1^g$. By renaming the sets and the radii, we have a new decomposition $\{A_i\}$ of $\Delta_Q\cap B_{1/8} (0)$, $i\in I_2$, and a new covering $\{B_{s_i} (x_i)\}$, $i\in I_2$, with
\[
\sum_{i\in I_2} s_i^{m-2} \leq CV \sum_{i\in I_1} s_i^2 \leq C_V^2\, .
\]
This time, however, if $s_i >\rho$, then the frequency drop is given by
\[
\sup \{ \freq\, (y,s_i): y \in A_i \} \leq U_0 - 2\delta\, .
\]
Proceeding inductively for each $k$ we find a decomposition $\{A_i\}_{i\in I_k}$ and corresponding covering $\{B_{s_i} (x_i)\}$ with the properties that
\[
\sum_{i\in I_k} s_i^{m-2} \leq C_V^k\, 
\]
and either $s_i =\rho$ or 
\[
\sup \{ \freq\, (y,s_i): y \in A_i \} \leq U_0 - k\delta\, .
\]
Clearly, since the frequency function is always positive, after at most $\kappa = \lfloor \delta^{-1} U_0 \rfloor+1$ steps all $s_i$ for $i\in I_\kappa$ equal $\rho$. 
We have thus found a family of $N$ balls $B_\rho (x_i)$ with  $N \rho^{m-2} \leq C_V^{\kappa} = C(m,n,Q,\Lambda)$ which cover $\Delta_Q \cap B_1 (0)$.
Obviously $B_\rho (\Delta_Q \cap B_{1/8} (0)) \subset \cup_i B_{2\rho} (x_i)$ and we thus conclude
\[
|B_\rho (\Delta_Q \cap B_{1/8} (0))|\leq 2^m N \rho^m \leq C \rho^2\, .\qedhere
\]
\end{proof}

\subsection{Intermediate covering} Proposition \ref{lemma_cover_final} will in fact be reached through an intermediate covering. 

\begin{lemma}\label{lemma_cover_intermediate}
Let $u$ be as in Assumptions \ref{a:primaria} and \ref{a:secondaria}, $\rho\leq 100^{-1}$ and $\sigma<\tau\leq \frac{1}{8}$ be three given positive numbers and $x\in B_{1/8} (0)$. Let $D$ be any subset of $\Delta_Q \cap B_{\tau} (0)$ and set $U\equiv \sup_{y\in D} \freq (y,\tau) $. 
Then there are a $\delta_{\ref{lemma_cover_intermediate}}=\delta(m,n,Q,\Lambda ,\rho)>0$, a constant $C=C_R(m)$ and a covering of $D$ by balls $\B {r_i}{x_i}$ with the following properties
\begin{itemize}
\item[(a)] $r_i \geq 10 \rho \sigma$;
\item[(b)] $\sum_{i\in I} r_i^{m-2}\leq C_R \tau^{m-2}$;
\item[(c)] For each $i$, either $r_i\leq \sigma$, or the set of points
 \begin{gather}
  F_i= D\cap \B {r_i}{x_i}\cap \{y: \freq (y,\rho r_i)>U-\delta \}
 \end{gather}
is contained in $\B {\rho r_i}{L_i}\cap \B {r_i}{x_i}$, for some $(m-3)$-dimensional affine subspace $L_i$.
\end{itemize}
\end{lemma}

\begin{proof} By a simple scaling and translation argument, from now on we can simply assume that $\tau = \frac{1}{8}$ and $x=0$. Observe that after this operation $\freq (0, 64)$ might have increased: anyway, according to Lemma \ref{l:limiti uniformi_II}, we will still be able to bound it in terms of $\Lambda$. For the rest of the argument we treat $\delta>0$ as fixed and detail the conditions that it will have to satisfy along the steps of the proof: we will see at the end that all such conditions are met if $\delta$ is chosen sufficiently small. 

The first part of the proof consists in constructing a first covering via an inductive procedure consisting of $\kappa =-  \lfloor\log_{10\rho} (8\sigma)\rfloor$ steps (note that $\kappa$ is the smallest integer exponent such that $8^{-1} (10\rho)^{\kappa} \leq \sigma$). At each step $k$ we will thus have a covering of $D$ by balls $\mathscr{C} (k) = \{B_{\rho_i} (x_i) : i \in I_k\}$. 
The starting cover is given by $\{B_{1/8} (0)\}$ and the cover $\mathscr{C} (k+1)$ is obtained by modifying $\mathscr{C} (k)$ suitably: in particular we keep some ``bad'' balls $B$ of $\mathscr{C} (k)$ in $\mathscr{C} (k+1)$ and we refine the covering on some other ``good'' balls $B$. Along this procedure we have the following conditions:
\begin{itemize}
\item[(i)] the radii of the balls in $\mathscr{C} (k)$ are  all equal to some $8^{-1} (10 \rho)^j$ with integer exponents $j$ ranging from $0$ to $k$;
\item[(ii)] if $B_{r} (x), B_{r'} (x')\in \mathscr{C} (k)$, then $B_{r/5} (x) \cap B_{r'/5} (x') = \emptyset$;
\item[(iii)] if a ball in $\mathscr{C} (k)$ has radius larger than $8^{-1} (10\rho)^{k}$, then it is certainly kept in $\mathscr{C} (k+1)$.
\end{itemize}

\medskip

{\bf Step 1. Inductive procedure.} Consider a ball $B_r (x)\in \mathscr{C} (k)$. If $r = 8^{-1} (10\rho)^{j}$ for some $j< k$, then we assign it to $\mathscr{C} (k+1)$. If $r= 8^{-1} (10\rho)^k$, consider the set
\[
F = F (B_r (x)) := D \cap B_r (x)\cap \{y: \freq (y, \rho r) > U-\delta \}\, .
\]
We then
\begin{itemize}
\item[(bad)] assign $B_r (x)$ to $\mathscr{C} (k+1)$ if $F$ does not $\rho r$-span an $(m-2)$-dimensional space;
\item[(good)] discard $B_r (x)$ if $F$ $\rho r$-spans an $(m-2)$-dimensional space, which we call $L = L (B_r (x))$.
\end{itemize}
We note first that, if (bad) holds, then there is an $(m-3)$-dimensional affine space $L$ such that $F\subset B_\rho (L)$, cf. Remark \ref{r:m-3}. 
If (good) holds, we must replace $B_r (x)$ in $\mathscr{C} (k+1)$ with a new collection $\{B_{10 \rho r} (x_i)\}$. 

More precisely, in the latter case consider an $(m-2)$-dimensional
affine space $V$ that is $\rho r$-spanned by $F$. By Lemma \ref{lemma_n-2_pinch}, if $\delta$ is chosen smaller than a constant $\bar\delta (m,n,Q, \Lambda, \rho)$, we can assume that $D\cap B_1 (0)$ is contained in $B_{\rho r} (V)$. Consider now all the good balls $\{B^i\} = \mathscr{G} (k)\subset \mathscr{C} (k)$, the corresponding affine spaces $V_i$ and the set 
\[
G (k) := D\cap \bigcup_i B_{\rho (10 \rho)^k} (V_i)\, .
\]
We can cover $G (k)$ with a collection $\mathscr{F} (k+1)$ of balls with radius $(10\rho)^{k+1}$ such that the corresponding concentric balls of radii $2 \rho (10\rho)^{k}$ are pairwise disjoint. It will also be important for the next step that such balls are chosen so that their centers are contained in $D\cap \ton{\cup_i B^i\cap V_i}$.

Consider now the collection $\mathscr{B} (k)\subset \mathscr{C} (k)$ of balls that have been kept in the covering $\mathscr{C} (k+1)$ and let $\mathscr{B}_{1/5} (k)$ be the corresponding collection of concentric balls shrunk by a factor $\frac{1}{5}$. We include $B\in \mathscr{F} (k+1)$ in the covering $\mathscr{C} (k+1)$ if and only if $B$ does not intersect any element of $\mathscr{B}_{1/5} (k)$. We need however to check that $\mathscr{C} (k+1)$ is still a covering of $D$. Consider that, by construction $\mathscr{B} (k) \cup \mathscr{F} (k+1)$ is certainly a covering of $D$. Pick a point $x\in D$: if it is contained in an element of $\mathscr{B} (k)$ we are fine. Otherwise it must be contained in an element $B$ of $\mathscr{F} (k+1)$. If $B$ is not contained in $\mathscr{C} (k+1)$, then there is a ball $B_{r'} (x')\in \mathscr{B} (k)$ such that $B_{r'/5} (x')$ intersects $B$. Since however the radius of $B$ is at most than $10 r'\leq r'/10$, it is obvious that $B$ is contained in $
 B_{r'} (x')$. 

\medskip

{\bf Step 2. Frequency pinching.} We next claim the following pinching estimate: for any given $\eta>0$, if we choose $\delta$ sufficiently small, then
\begin{equation}\label{e:pinching_new}
\mbox{either  $\mathscr{C} (\kappa) = \{B_{1/8} (0)\}$}\quad\mbox{ or }\quad \freq (x,  \rho s/5)\geq U-\eta \qquad \forall B_s (x) \in \mathscr{C} (k)\, .
\end{equation}
Indeed, unless the refining procedure stops immediately, for any $B_s (x)\in \mathscr{C} (k)$ we must have $s = 8^{-1} (10\rho)^{j+1}$ for some $j\in \mathbb N$. Following our construction, we then find a good ball $B' = B_{8^{-1} (10\rho)^j} \in \mathscr{C} (j)$ such that $F (B')$ $8^{-1}\rho (10\rho)^j$-spans an $(m-2)$-dimensional affine space $V$ with $x\in V\cap B'$. Moreover $V\cap B'$ contains at least one point $z\in F (B')$. It then follows from Lemma \ref{lemma_N_span} that, if we choose $\delta$ sufficiently small (depending on $\rho$ and $\eta$), then we can ensure
\[
|\freq (x, \rho s/5) - \freq (z, s)|\leq \frac{\eta}{2}\, .
\]
Since  however $\freq (z, s) \geq U - \delta$, the claim follows by imposing additionally $\delta < \frac{\eta}{2}$. 

\medskip

{\bf Step 3. Discrete measures.} The covering of the statement of the lemma is now given by $\mathscr{C} (\kappa)$ and it is clear that to complete the proof it just suffices to prove the packing bound
\[
\sum_{B_s (x) \in \mathscr{C} (\kappa)} s^{m-2} \leq C_R (m)\, .
\]
For this reason, from now we enumerate the balls in $\mathscr{C} (\kappa)$ as $B_{5s_i} (x_i)$, $i\in I$. Since our goal is to use Theorem \ref{th_reif_vol}, we introduce the measures 
\[
\mu = \sum_{i\in I} s_i^{m-2} \delta_{x_i} \qquad \mbox{and} \qquad \mu_s = \sum_{i\in I, r_i \leq s} s_i^{m-2} \delta_{x_i}\, .
\]
Observe that:
\begin{itemize}
\item $\mu_t \leq \mu_\tau$ if $t\leq \tau$;
\item $\mu = \mu_{1/40}$;
\item if we define $\bar{r} = \frac{1}{40} (10\rho)^\kappa$, then $\mu_s=0$ for $s<\bar r$.
\end{itemize}
We will show that $\mu_s (B_s (x)) \leq C_R (m) s^{m-2}$ for every $s$ and every $x$. Indeed, if we set $\varkappa= \log_2 (\bar r^{-1}/8) -4$, it suffices to show that
\begin{equation}\label{e:packing_induttivo}
\mu_s (B_s (x)) \leq C_R (m)s ^{m-2} \qquad \mbox{for all $x$ and for all $s= \bar r 2^j$ with $j=0, 1, 2, \ldots, \varkappa$.}
\end{equation}
Note indeed that, unless $\{B_{s_i} (x_i)\}$ is the trivial cover $\{B_{1/8} (0)\}$, all the radii $s_i$ are smaller than $\frac{10 \rho}{40} \leq \frac{1}{400}$ and thus \eqref{e:packing_induttivo} shows that $\mu (B_{1/128} (x)) \leq C_R (m)$  for every $x\in B_{1/8} (0)$. Covering $B_{1/8} (0)$ with finitely many balls of radius $\frac{1}{128}$ implies then the desired packing estimate. 

The estimate \eqref{e:packing_induttivo} will be proved by induction over $j$. Note that the starting step is fairly easy. Indeed, $\mu_{\bar r} (B_{\bar r} (x)) = N (x,\bar r) \bar r^{m-2}$, where $N (x,\bar r)$ is the number of balls $B_{s_i} (x_i)$ with $s_i = \bar r$ and $x_i \in B_{\bar r} (x)$.
Since such balls are pairwise disjoint and contained in $B_{2\bar r} (x)$, the number, $N (x,s)$ is bounded by $2^m$.

The remaining portion of the proof is devoted to show that if \eqref{e:packing_induttivo} holds for some $j< \varkappa$ then it holds for $j+1$. Hence from
now on we set $r= 2^j \bar r$ and, assuming $\mu_r (B_r (x) \leq C_R (m) r^{m-2}$ for every $x$, we want to show that 
$\mu_{2r} (B_{2r} (x))\leq C_r (m) (2r)^{m-2}$ for every $x$.

\medskip

{\bf Step 4. Inductive packing estimate: coarse bound.} We first show the coarser bound
\begin{equation}\label{e:coarser}
\mu_{2r} (B_{2r} (x)) \leq C(m) C_R (m) (2r)^{m-2}\, ,
\end{equation}
where $C_R (m)$ is a dimensional constant larger than $1$. This is rather easy to achieve since we can split
\[
\mu_{2r} = \mu_r + \sum_{i\in I, r< s_i  \leq 2r} s_i^{m-2} \delta_{x_i}=: \mu_r + \tilde{\mu}_r\, .
\]
Since $B_{2r} (x)$ can be covered by $C(m)$ balls $B_r (x_i)$, the inductive assumption clearly implies 
\[
\mu_{r} (B_{2r} (x)) \leq C(m) C_R (m) r^{m-2}\, .
\] 
On the other hand $\tilde{\mu_r} (B_{2r} (x)) \leq N (x, 2r) (2r)^{m-2}$, where $N (x,2r)$ is the number of balls $B_{s_i}$ with $i\in I$, $r< s_i \leq 2r$ and $x_i \in B_{2r} (x)$. The corresponding smaller balls $B_r (x_i)$ are then all pairwise disjoint and contained in $B_{3r} (x)$, from which the bound $N (x,r) \leq C (m)$ follows readily. 

\medskip

{\bf Step 5. Inductive packing estimate: mean flatness and conclusion} We now wish to improve the coarse bound \eqref{e:coarser} to 
\begin{equation}\label{e:vero_bound}
\mu_{2r} (B_{2r} (x)) \leq C_R (m) (2r)^{m-2}\, .
\end{equation}
We set for convenience $\bar \mu := \mu_{2r} \res B_{2r} (x)$. The idea is to apply a (scaled version) of Theorem \ref{th_reif_vol}. If we can show that
\begin{gather}\label{D_reif_hyp}
 \int_{\B {t} y } \ton{\int_0^t D^{m-2}_{\bar \mu} (z,s)\,{\frac{ds}{s}}}\, d\bar \mu(z)<\delta_0^2 t^{m-2} \qquad \forall y\in B_{2r} (x), \forall 0<t\leq 2r\, 
\end{gather}
(where $\delta_0$ is the constant of Theorem \ref{th_reif_vol}), we will then conclude $\bar\mu (B_{2r} (x)) \leq C_R (2r)^{m-2}$, which is the desired bound. 

The key for deriving  \eqref{e:vero_bound} is that, by \eqref{e:pinching_new}, we can, without loss of generality, assume 
\begin{equation}\label{e:pinching_2}
I_\phi (x_i, \rho s_i) \geq U -\eta\, .
\end{equation}
In fact if this estimate did not hold the covering $\{B_{s_i} (x_i)\}$ would be given by $\{B_{1/8} (0)\}$ and the claim \eqref{e:packing_induttivo} would be trivially true. 

In order to obtain the bound \eqref{D_reif_hyp}, we first set
\begin{equation}
\bar W_s (x_i) := \left\{
\begin{array}{ll}
W_{\sfava{s}}^{32s} (x_i) = \sfava{\freq (x_i, 32 s) - \freq (x_i, s)}\qquad &\mbox{if $s>s_i$}\\ \\
0 & \mbox{otherwise,}
\end{array}\right.
\end{equation}
and then observe that for all $i$
\begin{equation}\label{e:stima_chiave}
D_{\bar \mu}^{m-2} (x_i,s) \leq C(m,n,Q,\Lambda) s^{-(m-2)} \int_{B_s (x_i)} \bar{W}_s (y)\, d\bar\mu (y)\qquad \mbox{for all $0<s<1$.}
\end{equation}
Indeed, if \sfava{$s<s_i$}, the above inequality reduces to $0=0$ because $\supp (\mu) \cap B_s (x_i) = \{x_i\}$. Otherwise, it follows from Proposition
\ref{p:mean-flatness vs freq}.

Fix any $t\leq 2r$. Using \eqref{e:stima_chiave} we bound
\begin{align}
I := \int_{\B {t} y } \ton{\int_0^t D^{m-2}_{\bar \mu} (z,s)\,{\frac{ds}{s}}}\, d\bar \mu(z) &\leq C \int_{\B {t} y } \int_0^t s^{1-m} \int_{B_s (z)} \bar{W}_s (\zeta)\, d\bar\mu (\zeta)\, ds\, d\bar\mu (z)\nonumber\\
&= C \int_0^t s^{1-m} \int_{\B {t} y} \int_{B_s (z)} \bar{W}_s (\zeta)\, d\bar\mu (\zeta)\, d\bar\mu (z)\, ds\label{e:Fubini1}\, .
\end{align}
In \eqref{e:Fubini1} we can certainly intersect the domains of integrations with $B_{2r} (x)$, since $\bar\mu$ vanishes outside.
We also claim that we can substitute $\bar\mu$ with $\mu_s$. First we look at the innermost integral: if $\zeta\in \supp (\bar \mu) \setminus \supp (\mu_s)$, then $\zeta= z_i$ for some $i\in I$ with $s_i >s$ and, by definition $\bar{W}_s (\zeta) =0$. As for the integral in $z$, if $z=z_i$ for some $i\in I$ with $s_i >s$, then $B_s (z)\cap \supp (\bar\mu)$ contains only $z$ and the innermost integrand vanishes because $\bar W_s (z) =0$. Substituting $\bar \mu$ with $\mu_s$ and applying again Fubini's Theorem, we can write
\begin{align}
I &\leq  C \int_0^t s^{1-m} \int_{B_{t+s} (y) \cap B_{2r} (x)} \bar{W}_s (\zeta)  \int_{B_s (\zeta)\cap B_{2r} (x)} d\mu_s (z)\, d\mu_s (\zeta)\, ds\, .
\end{align}
Next, for $s\leq r$ we can use the inductive estimate \eqref{e:packing_induttivo}, whereas for $r\leq s \leq 2r$ we can use the coarser bound 
\eqref{e:coarser} to estimate the inner integrand with $C (m) s^{m-2}$. We therefore achieve
\begin{align}
I &\leq C (m, n, Q, \Lambda) \int_0^t \int_{B_{t+s} (y) \cap B_{2r} (x)} \bar{W}_s (\zeta)\, d\mu_s (\zeta)\, \frac{ds}{s}
\leq C \int_0^t \int_{B_{t+s} (y) \cap B_{2r} (x)} \bar{W}_s (\zeta)\, d\mu_t\, \frac{ds}{s}\nonumber\\
&\leq C\int_{B_{2t} (y)} \int_0^t \bar{W}_s (\zeta)\, \frac{ds}{s}\, d\mu_t (\zeta)\, .\label{e:mean_flat_est}
\end{align}
Next fix $\zeta\in \supp (\mu_t)$. Then obviously $\zeta = z_i$ for some $i$. Recall that $\bar{W}_s (z_i) =0$ if $s<s_i$ and
that $\bar W_s (z_i) = \freq (z_i, 32 s) - \sfava{\freq (z_i, s)}$ otherwise. Consider 
now the largest integer $\kappa$ such that $2^\kappa s_i\geq t$ and note that $32 \cdot 2^{\kappa+1} s_i < \frac{1}{8}$. 
Then we can derive the following estimate
\begin{align}
\int_0^t \bar{W}_s (\zeta)\, \frac{ds}{s} = & \int_{s_i}^t \bar{W}_s (z_i)\, \frac{ds}{s} = \int_{s_i}^t (\freq (z_i, 32 s) - \freq (z_i, 2s))\, \frac{ds}{s}\nonumber\\
\leq & \sum_{j=0}^\kappa \int_{2^j{s_i}}^{2^{j+1} s_i} (\freq (z_i, 32 s) - \sfava{\freq (z_i, s)})\, \frac{ds}{s}\nonumber\\
\leq & \sum_{j=0}^\kappa (\freq (z_i, 32\cdot 2^{j+1} s_i) - \sfava{\freq (z_i,  2^j s_i))} \int_{2^j{s_i}}^{2^{j+1} s_i}\frac{ds}{s}\nonumber\\
= & \log 2 \sum_{j=0}^\kappa (\freq (z_i, 2^{6+j} s_i) - \sfava{\freq (z_i, 2^{j} s_i))}\nonumber\\ 
= & \log 2 \sum_{\sfava{\ell=0}}^{5} \sum_{j=0}^\kappa (\freq (z_i, 2^{j+\ell+1} s_i) - \freq (z_i, 2^{j+\ell} s_i))\nonumber\\
=& \log 2 \sum_{\sfava{\ell=0}}^5 (\freq (z_i, 2^{\kappa+\ell+1}s_i) - \freq (z_i, 2^\ell s_i))\nonumber\\
 \leq& \sfava{6} \log 2 (\freq (z_i, {\textstyle{\frac{1}{8}}}) - \freq (z_i, s_i))\stackrel{\eqref{e:pinching_new}}{\leq} \sfava{6} \eta \log 2\, .\label{e:telescopica}
\end{align}
Next, with an obvious covering argument we can use the inductive estimate \eqref{e:packing_induttivo} (for $t\leq r$) and the coarser estimate \eqref{e:coarser} (in the case $r<t\leq 2r$), to estimate $\mu_t (B_{2t} (y)) \leq C (m) t^{m-2}$. Combined with \eqref{e:telescopica}, the latter bound in \eqref{e:mean_flat_est} yields 
 \begin{align}
\int_{\B {t} y } \ton{\int_0^t D^{m-2}_{\bar \mu} (z,s)\,{\frac{ds}{s}}}\, d\bar \mu(z) &\leq C (m,n,Q, \Lambda)\,  \eta\, t^{m-2}\, .
\end{align}
At this point, choosing $\eta$ smaller than some appropriate constant $c (m,n,Q, \Lambda)$ (which requires $\delta$ to be chosen smaller than a suitable positive constant $c(m,n,Q, \Lambda, \rho)$) allows us to fulfill \eqref{D_reif_hyp} and thus complete the proof of
\eqref{e:packing_induttivo}. 
\end{proof}

\subsection{Proof of Proposition \ref{lemma_cover_final}} As in the proof of the previous lemma, we start by observing that without loss of generality we can assume $x=0$ and $r=\frac{1}{8}$. The proof of the Proposition is again an inductive procedure to generate the correct covering, where we use Lemma \ref{lemma_cover_intermediate}. The parameter $\rho$ appearing in the Lemma is, for the moment, fixed: it will be chosen, sufficiently small, only at the end.

\medskip

We start by applying Lemma \ref{lemma_cover_intermediate} a first time with $\tau=\frac{1}{8}$ and $\sigma = s$. Let $\mathcal{C} (0) = \{B_{r_i} (x_i)\}$ be the corresponding covering. We then divide $\mathcal{C} (0)$ as $\mathcal{G} (0) = \{B_{r_i} (x_i): r_i \leq s\}$ and $\mathcal{B} (0) =
\{B_{r_i} (x_i): r_i >s\}$. Next, for each $B_{r_i} (x_i) \in \mathcal{B} (0)$ consider the set $F_i$ and the affine plane $L_i$ given by Lemma \ref{lemma_cover_intermediate}. Each $B_{2\rho_i r_i} (L_i)\cap B_{r_i} (x_i)$ can be covered by a number $N \leq C (m) \rho^{3-m}$ of balls of radius $4 \rho r_i$. If $4 \rho r_i <s$ we then
include these balls in a new (additional) collection $\mathcal{C} (1)$. Otherwise we apply to each of these balls and for each $i$ Lemma \ref{lemma_cover_intermediate} again and include all these balls in the new collection $\mathcal{C} (1)$. Observe that we have the bound
\[
\sum_{B_{r_i} (x_i) \in \mathcal{C} (1)} r_i^{m-2} \leq C (m) \rho^{3-m} \sum_{B_{r_j} (x_j) \in \mathcal{C} (0)} ( \rho r_j)^{m-2}
= C(m) \rho \sum_{B_{r_j} (x_j) \in \mathcal{C} (0)} r_j^{m-2}\, .
\]
In particular if $\rho$ is chosen sufficiently small, we can ensure that
\begin{gather}\label{eq_rho}
C (m)\rho \leq \frac{1}{2} \quad \Longleftrightarrow \quad \rho \leq (2C(m))^{-1} :=\rho_0(m)\, .
\end{gather}
We repeat the procedure finitely many times
until we find a $\mathcal{C} (k)$ that contains no balls of radius larger than $s$. We then define the collection $\mathcal{C} = \cup_{j\leq k} \mathcal{C} (j)$. Clearly
\[
\sum_{B_{r_i} (x_i)\in \mathcal{C}} r_i^{m-2} \leq \sum_{\ell=0}^k 2^{-\ell} \sum_{B_{r_j} (x_j) \in \mathcal{C} (0)} r_j^{m-2} \leq 2 C_R (m)\, .
\]
From now on $\rho$ is fixed, depending only on the dimension $m$. 

We then define inductively the sets $A'_i$ for each $B_{r_i} (x_i)\in \mathcal{C}$. 
We start with the elements $\mathcal{C} (0)$: 
\begin{itemize}
\item if $B_{r_i} (x_i)\in \mathcal{B} (0)$, namely $r_i \leq s$, we then set $A'_i = D \cap B_{r_i} (x_i)$;
\item otherwise we set $A'_i = (D\cap B_{r_i} (x_i))\setminus F_i$, where $F_i$ are the sets of Lemma \ref{lemma_cover_intermediate}. 
\end{itemize}
Observe that by construction the $F_i$'s are covered by $\mathcal{C} (1)$ and thus
\[
D \subset \bigcup_{B_{r_i} (x_i)\in \mathcal{C} (0)} A'_i \bigcup_{B_{r_i} (x_i)\in \mathcal{C} (1)} B_{r_i} (x_i)\, .
\]
we then proceed inductively and notice that at the final step all balls of $\mathcal{C} (k)$ have radii no larger than $s$. Thus the final collection of sets $A'_i$ is a covering of $D$.

Moreover, by definition, either $r_i\leq s$, or 
\begin{gather}
\sup \{ \freq\, (y,\rho r_i): y \in A'_i \} \leq U- \delta\, .\notag
\end{gather}
This condition differs from \eqref{eq_Edrop} just by a factor of $\rho=\rho(m)$ inside the frequency $\freq$. Since $A'_i\subseteq \B {s_i}{x_i}$, we can clearly cover this set by a family of $C(m)\rho^{-m} = C(m)$ balls $\B {\rho s_i}{x_{ij}}$ (recall that $\rho$ has already been fixed as a positive geometric constant depending only on $m$ in \eqref{eq_rho}). By setting $A_{ij}= \B {\rho s_i}{x_{ij}}\cap A'_i$, we get \eqref{eq_Edrop} on this set, and preserve up to a constant $C(m)$ the packing estimate.

Finally, some of the balls in $\mathcal{C}$ have radii strictly smaller than $s$. However by construction they are all larger than $10\rho s$. Hence we can substitute such balls with balls of radius $s$ at the price of paying another multiplicative constant $C(m)$ in the packing estimate.

\section{Rectifiability} 

In this section we complete our plan by giving a proof of Theorem \ref{t:Rect}. 
The crucial ingredient is the content of \cite[Corollary 1.3]{AzzTol}, which we cite here without proof. 

\begin{theorem}[{\cite[Corollary 1.3]{AzzTol}}] \label{th_rect}
 Let $S\subset \R^n$ be $\cH^k$-measurable with $\cH^k(S)<\infty$ and consider $\mu=\cH^k\res S$. Then $S$ is countably $k$-rectifiable if and only if
 \begin{gather}\label{eq_bound_rect}
  \int_0^1 D^k_{\mu}(x,s)\frac{ds}{s}<\infty\qquad \mbox{for $\mu$-a.e. $x$.}
 \end{gather}
\end{theorem}
Using a different proof, a similar result was obtained in \cite[Theorem 3.3]{NV}, which in some sense is the ``continuous version'' of Theorem \ref{th_reif_vol}. Indeed, the rectifiability result is a corollary of the proof of Theorem \ref{th_reif_vol}, since in order to obtain the uniform bounds for the measure $\mu$ one needs to build smooth manifolds that approximate the measure $\mu$ at smaller and smaller scales. If instead of a discrete measure $\mu$ one considers the $k$-dimensional Hausdorff measure $\cH^k$ restricted to a set $S$, the construction basically works in the same way and produces a Lipschitz approximation for $S$ that coincides with $S$ up to a set of small measure. By repeating this construction inductively, one proves rectifiability.

Notice also that in order to obtain the estimate \eqref{eq_bound_rect}, we will need to use the uniform upper Ahlfors bounds on the measure $\cH^k\res \Delta_Q$, which is the main product of our construction, and the main point of Theorem \ref{th_reif_vol}. With this uniform estimate in hand, it is easier to apply directly Theorem \ref{th_rect} instead of going through the details of \cite[Theorem 3.3]{NV}.
% 
% 
% The main difference between this theorem and \cite[Theorem 3.3]{NV} is that in this second version the hypothesis is stronger. In particular, \cite[Theorem 3.3]{NV} requires \eqref{eq_bound_rect} to have a \textit{small} bound $\delta(n)$, not just finiteness. On the other hand, \cite[Theorem 3.3]{NV} also gives uniform upper Ahlfors bounds on the measure $\cH^k\res S$. This bound is essential for our estimates. However, since we already have proved upper Ahlfors bounds on $\cH^k\res \Delta_Q$ as a corollary of Theorem \ref{t:Minkio}, it is simpler to use Azzam and Tolsa's result to prove rectifiability for the singular set.

\begin{proof}[Proof of Theorem \ref{t:Rect}]
We know from Theorem \ref{t:Minkio} that $\mu = \mathcal{H}^{m-2} \res (\Delta_Q \cap B_{1/8})$ is a finite Radon measure. But in fact, by a simple scaling argument, we achieve the uniform estimate
\begin{equation}\label{e:uniform}
\mu (B_r (x)) \leq C (m,n,Q, \Lambda)r^{m-2}\, .
\end{equation}
As in the last step of the proof of Lemma \ref{lemma_cover_intermediate} we use Proposition \ref{p:mean-flatness vs freq} to estimate
\begin{align}
\int_{B_t (y)} \int_0^t D^{m-2}_\mu (z,s)\, \frac{ds}{s}\, d\mu (z) \leq & C \int_{B_t (y)} \int_0^t s^{1-m} \int_{B_s (z)} \sfava{W_{s}^{32s}} (\zeta)\, d\mu (\zeta)\, ds\, d\mu (z)\nonumber\\
= & C  \int_0^t s^{1-m} \int_{B_t (y)} \int_{B_s (z)} \sfava{W_{s}^{32s}} (\zeta)\, d\mu (\zeta)\, d\mu (z)\, ds\nonumber\\
\leq & C  \int_0^t s^{1-m} \int_{B_{t+s} (y)} \sfava{W_{s}^{32s}} (\zeta) \int_{B_s (\zeta)}\, d\mu (\zeta)\, d\mu (z)\, ds\nonumber\\
\stackrel{\eqref{e:uniform}}{\leq} & C \int_0^t s^{-1} \int_{B_{t+s} (y)} \sfava{W_{s}^{32s}} (\zeta) \, d\mu (\zeta)\, ds\nonumber\\
\leq & C \int_{B_{2t} (y)} \int_0^t \sfava{W_{s}^{32s}} (\zeta) \, \frac{ds}{s}\, d\mu (\zeta)\, .\label{e:ancora_mean}
\end{align}
Next arguing as in the proof of \eqref{e:telescopica}, we reach
\[
\int_0^t \sfava{W_{s}^{32s}} (\zeta) \, \frac{ds}{s} \leq \sfava{6} \log 2 (\freq (\zeta, \textstyle{\frac{1}{8}}) - \freq (\zeta, 0)) \leq C (m,n,Q, \Lambda)\, ,
\]
as long as $32 t < \frac{1}{8}$. 
Inserting the latter estimate in \eqref{e:ancora_mean} and using \eqref{e:uniform} we then conclude
\[
\int_{B_t (y)} \int_0^t D^{m-2}_\mu (z,s)\, \frac{ds}{s}\, d\mu (z) < \infty\, ,
\]
whenever $t< \frac{1}{8}\cdot \frac{1}{32}$. 
%We can thus apply a (suitably scaled version of) Theorem \ref{th_rect} together with an obvious covering argument to conclude the rectifiability of $\Delta_Q \cap B_{1/8} (0)$. 
We can thus apply Theorem \ref{th_rect} to conclude the rectifiability of $\Delta_Q \cap B_{1/8} (0)$. 
\end{proof}

\Addresses

%%%%%%%%%%%%%%%%%%%%%%%%%%%%%%%%%%%%%%%%%%%%%%%%%%%%%%%%%%%%%%%%%
%
%	BIBLIOGRAPHY
%
%%%%%%%%%%%%%%%%%%%%%%%%%%%%%%%%%%%%%%%%%%%%%%%%%%%%%%%%%%%%%%%%%

% \bibliographystyle{plain}
% \begin{thebibliography}{99}
% \setlength{\itemsep}{3pt}
% \newcommand{\aut}[1]{\textsc{#1}}
% 
% 
% 
% \bibitem{DS0}
% \aut{C.~De Lellis, E.~Spadaro.}
% \newblock Q-valued functions revisited.
% \newblock \textit{Mem. Amer. Math. Soc.}, 211 (2011), no. 991, vi+79 pp.
% 
% \bibitem{DS5}
% \aut{C.~De Lellis, E.~Spadaro.}
% \newblock Regularity of area minimizing currents III: blow-up.
% \newblock \textit{Ann. of Math. (2)}, 183 (2016), no. 2, 577-617. 
% 
% \bibitem{FMS}
% \aut{M.~Focardi, A.~Marchese, E.~Spadaro.}
% \newblock Improved estimate of the singular set of Dir-minimizing Q-valued
% functions via an abstract regularity result.
% \newblock \textit{J. Funct. Anal.}, 268 (2015), no. 11, 3290-3325.
% 
% \bibitem{NV}
% \aut{A.~Naber, D.~Valtorta.}
% \newblock Rectifiable-Reifenberg and the regularity of stationary and minimizing harmonic maps.
% \newblock \textit{Preprint. Available at http://arxiv.org/abs/1504.02043.}
% 
% \bibitem{hanlin}
% \aut{Quin Han, Fang-Hua Lin}
% \newblock Nodal sets of solutions of elliptic differential equations.
% \newblock \textit{Available at \href{http://nd.edu/~qhan/nodal.pdf}{http://nd.edu/~qhan/nodal.pdf}}
% 
% \end{thebibliography}

% Propongo di usare Bibtex per la bibliografia. La vecchia bibliografia rimane commentata qui sopra.

\bibliographystyle{aomalpha}
\bibliography{Q-valued_bib}

\end{document}